\documentclass{conm-p-l}
\usepackage{amssymb}
\usepackage{amsfonts}
\usepackage{eucal}
\begin{document} 
\def\bbR{\mathrm{I\!R}}
\def\bbZ{\mathsf{Z\hskip-4ptZ}}
\def\bbC{{\mathchoice {\setbox0=\hbox{$\displaystyle\rm C$}\hbox{\hbox
to0pt{\kern0.4\wd0\vrule height0.9\ht0\hss}\box0}}
{\setbox0=\hbox{$\textstyle\rm C$}\hbox{\hbox
to0pt{\kern0.4\wd0\vrule height0.9\ht0\hss}\box0}}
{\setbox0=\hbox{$\scriptstyle\rm C$}\hbox{\hbox
to0pt{\kern0.4\wd0\vrule height0.9\ht0\hss}\box0}}
{\setbox0=\hbox{$\scriptscriptstyle\rm C$}\hbox{\hbox
to0pt{\kern0.4\wd0\vrule height0.9\ht0\hss}\box0}}}}
\def\bbQ{{\mathchoice {\setbox0=\hbox{$\displaystyle\rm Q$}\hbox{\raise
0.15\ht0\hbox to0pt{\kern0.4\wd0\vrule height0.8\ht0\hss}\box0}}
{\setbox0=\hbox{$\textstyle\rm Q$}\hbox{\raise
0.15\ht0\hbox to0pt{\kern0.4\wd0\vrule height0.8\ht0\hss}\box0}}
{\setbox0=\hbox{$\scriptstyle\rm Q$}\hbox{\raise
0.15\ht0\hbox to0pt{\kern0.4\wd0\vrule height0.7\ht0\hss}\box0}}
{\setbox0=\hbox{$\scriptscriptstyle\rm Q$}\hbox{\raise
0.15\ht0\hbox to0pt{\kern0.4\wd0\vrule height0.7\ht0\hss}\box0}}}}
\def\bbQ{{\mathchoice {\setbox0=\hbox{$\displaystyle\rm Q$}\hbox{\raise
0.15\ht0\hbox to0pt{\kern0.4\wd0\vrule height0.8\ht0\hss}\box0}}
{\setbox0=\hbox{$\textstyle\rm Q$}\hbox{\raise
0.15\ht0\hbox to0pt{\kern0.4\wd0\vrule height0.8\ht0\hss}\box0}}
{\setbox0=\hbox{$\scriptstyle\rm Q$}\hbox{\raise
0.15\ht0\hbox to0pt{\kern0.4\wd0\vrule height0.7\ht0\hss}\box0}}
{\setbox0=\hbox{$\scriptscriptstyle\rm Q$}\hbox{\raise
0.15\ht0\hbox to0pt{\kern0.4\wd0\vrule height0.7\ht0\hss}\box0}}}}
\def\aff{\mathrm{A\hn f\hh f}\hs}
\def\hyp{\hskip.5pt\vbox
{\hbox{\vrule width3ptheight0.5ptdepth0pt}\vskip2.2pt}\hskip.5pt}
\def\dz{\mathcal{D}}
\def\dzp{\dz^\perp}
\def\fv{\mathcal{F}}
\def\wv{\mathcal{W}}
\def\vf{\varPhi}
\def\ep{E}
\def\zr{\mathcal{R}}
\def\vr{\mathcal{V}}
\def\cs{\mathcal{B}}
\def\zs{\mathcal{S}}
\def\vb{\mathcal{Z}}
\def\zb{\hat{\mathcal{Z}}}
\def\g{\mathtt{g}}
\def\lz{\mathcal{L}}
\def\xe{\mathcal{E}}
\def\hga{\hskip1.9pt\widehat{\hskip-1.9pt\gamma\hskip-1.6pt}\hskip1.6pt}
\def\hm{\hskip1.9pt\widehat{\hskip-1.9ptM\hskip-.2pt}\hskip.2pt}
\def\hu{\hskip1.9pt\widehat{\hskip-1.9ptU\hskip-.2pt}\hskip.2pt}
\def\hg{\hskip.9pt\widehat{\hskip-.9pt\g\hskip-.9pt}\hskip.9pt}
\def\hna{\hskip.2pt\widehat{\hskip-.2pt\nabla\hskip-1.6pt}\hskip1.6pt}
\def\hdz{\hskip.9pt\widehat{\hskip-.9pt\dz\hskip-.9pt}\hskip.9pt}
\def\hdp{\hskip.9pt\widehat{\hskip-.9pt\dz\hskip-.9pt}\hskip.9pt^\perp}
\def\w{^{\phantom i}}
\def\r{k}
\def\y{y}
\def\vp{{\tau\hskip-4.55pt\iota\hskip.6pt}}
\def\bc{C}
\def\mv{V}
\def\hs{\hskip.7pt}
\def\hh{\hskip.4pt}
\def\hn{\hskip-.4pt}
\def\nh{\hskip-.7pt}
\def\nnh{\hskip-1pt}
\def\op{\varTheta}
\def\gy{\lambda}
\def\gp{\mathrm{G}}
\def\hp{\mathrm{H}}
\def\Gm{\Gamma}
\def\lr{\langle\hh\cdot\hs,\hn\cdot\hh\rangle}
\def\ve{\varepsilon}
\def\yt{\eta}

\newtheorem{theorem}{Theorem}[section] 
\newtheorem{proposition}[theorem]{Proposition} 
\newtheorem{lemma}[theorem]{Lemma} 
\newtheorem{corollary}[theorem]{Corollary} 
  
\theoremstyle{definition} 
  
\newtheorem{defn}[theorem]{Definition} 
\newtheorem{notation}[theorem]{Notation} 
\newtheorem{example}[theorem]{Example} 
\newtheorem{conj}[theorem]{Conjecture} 
\newtheorem{prob}[theorem]{Problem} 
  
\theoremstyle{remark} 
  
\newtheorem{remark}[theorem]{Remark}

\renewcommand{\theequation}{\arabic{section}.\arabic{equation}}

\title[Rank-one ECS manifolds of di\-la\-tion\-al type]{Rank-one ECS manifolds
of di\-la\-tion\-al type}
\author[A. Derdzinski]{Andrzej Derdzinski} 
\address{Department of Mathematics, The Ohio State University, 
Columbus, OH 43210, USA} 
\email{andrzej@math.ohio-state.edu} 
\author[I.\ Terek]{Ivo Terek} 
\address{Department of Mathematics, The Ohio State University, 
Columbus, OH 43210, USA} 
\email{terekcouto.1@osu.edu} 
\thanks{The first author's research was supported in part by a 
FAPESP\hn-\hh OSU 2015 Regular Research Award (FAPESP grant: 2015/50265-6). 
The authors wish to thank the anonymous referee, whose suggestions 
allowed us to improve the exposition.}

\subjclass[2020]{Primary 53C50}
\def\leftmark{A.\ Derdzinski \&\ I.\ Terek}
\def\rightmark{Rank-one ECS manifolds of di\-la\-tion\-al type}

\begin{abstract}
We study ECS manifolds, that is, 
pseu\-do\hs-Riem\-ann\-i\-an manifolds with parallel Weyl tensor which are
neither con\-for\-mal\-ly flat nor locally symmetric. Every ECS manifold has
rank 1 or 2, the rank being the dimension of a distinguished null parallel
distribution discovered by Ol\-szak, and a rank-one ECS manifold may be called
translational or di\-la\-tion\-al, depending on whether the holonomy group of
a natural flat connection in the Ol\-szak distribution is finite or infinite. 
Some such manifolds are in a natural sense generic, which refers to the
algebraic structure of the  Weyl tensor. Various examples of compact rank-one
ECS manifolds are known: translational ones 
(both generic and non\-ge\-ner\-ic) in every dimension $\,n\ge5$, as well as
odd-di\-men\-sion\-al non\-ge\-ner\-ic di\-la\-tion\-al ones, some of which
are locally homogeneous. As we 
show, generic compact rank-one ECS manifolds must be translational or locally
homogeneous, provided that they arise as isometric quotients of a specific
class of explicitly constructed ``model'' manifolds. This result is relevant
since the clause starting with ``provided that'' may be dropped: according to
a theorem which we prove in another paper, the models just mentioned
include the isometry types of the pseu\-\hbox{do\hskip.7pt-}Riem\-ann\-i\-an 
universal coverings of all generic compact rank-one ECS manifolds.
Consequently, all generic compact rank-one ECS manifolds are translational.
\end{abstract}

\maketitle

\setcounter{section}{0}
\setcounter{theorem}{0}
\renewcommand{\thetheorem}{\Alph{theorem}}
\section*{Introduction}
\setcounter{equation}{0}
By {\it ECS manifolds\/} \cite{derdzinski-roter-07} one means those 
pseu\-\hbox{do\hskip.7pt-}Riem\-ann\-i\-an manifolds 
of dimensions $\,n\ge4\,$ which have parallel Weyl tensor, but not for one of
two obvious reasons: con\-for\-mal flatness or local symmetry. Both their  
existence, for every $\,n\ge4$, and indefiniteness of their metrics, are
results of Roter \cite[Corol\-lary~3]{roter},
\cite[Theorem~2]{derdzinski-roter-77}. Their local structure has been 
completely described in \cite{derdzinski-roter-09}.

The acronym `ECS' stands for {\it essentially con\-for\-mal\-ly symmetric}. 
On every ECS manifold $\,(M\nh,\g)\,$ there exists a naturally distinguished
null parallel distribution $\,\dz$, known as the {\it Ol\-szak 
distribution\/} \cite{olszak}, \cite[p.\ 119]{derdzinski-roter-09}. Its
dimension, necessarily equal to $\,1\,$ or $\,2$, is referred to 
as the {\it rank\/} of $\,(M\nh,\g)$. We call a rank-one ECS manifold 
{\it translational}, or {\it di\-la\-tion\-al}, when the holonomy group of the
flat connection in $\,\dz$, induced by the Le\-vi-Ci\-vi\-ta connection, is 
finite or, respectively, infinite.

Examples of {\it compact\/} rank-one ECS manifolds are known 
\cite{derdzinski-roter-10,derdzinski-terek-ne} to exist for every dimension 
$\,n\ge5$. They are all geodesically complete, translational, and none of them
is locally homogeneous. Quite recently \cite{derdzinski-terek-cl} we
constructed di\-la\-tion\-al-type compact rank-one ECS manifolds, including
lo\-cal\-ly-ho\-mo\-ge\-ne\-ous ones, in all odd dimensions $\,n\ge5$. 
It remains an open question whether a compact ECS manifold may
have rank two, or be of dimension four.

In Section~\ref{ro} we describe specific rank-one ECS {\it model manifolds\/}
\cite[p.\ 93]{roter}, representing all dimensions $\,n\ge4\,$ and all
indefinite metric signatures. Some of them are {\it generic}, which refers to
a self-ad\-joint linear en\-do\-mor\-phism $\,A\,$ of a
pseu\-\hbox{do\hs-}Euclid\-e\-an vector space used in constructing the model
manifold, and means that there are only finitely many linear iso\-metries
commuting with $\,A$. (In Remark~\ref{gnric} we point out that this
genericity 
is an intrinsic geometric property of the metric, and not just a condition 
imposed on the construction.)

The di\-la\-tion\-al 
examples of
\cite{derdzinski-terek-cl}, mentioned earlier, are all non\-ge\-ner\-ic,
while among the translational ones in 
\cite{derdzinski-roter-10,derdzinski-terek-ne}, some are generic, and others
are not, which raises an obvious question: Can a di\-la\-tion\-al-type compact
rank-one ECS manifold be generic? 
Theorem~\ref{maith} of the present paper, combined with results of
\cite{derdzinski-terek-ms} mentioned below, answers this question in the
negative:
\begin{equation}\label{all}
\mathrm{all\ generic\ compact\ rank}\hyp\mathrm{one\ ECS\ manifolds\
are\ translational.}
\end{equation}
Here are some details. 
Since the Ol\-szak distribution $\,\dz\,$ is a {\it real line bundle\/} over
the compact rank-one ECS manifold in question, the holonomy group $\,K\hs$ of
the flat connection in $\,\dz\,$ induced by the Le\-vi-Ci\-vi\-ta connection
is a countable multiplicative sub\-group of $\,\bbR\smallsetminus\{0\}\,$ (see 
Section~\ref{pr}), and we will repeatedly refer to
\begin{equation}\label{ppr}
\mathrm{the\ positive\ holonomy\ group\ }\,K\hskip-1.5pt_+\w
\nnh\nh=\hn K\nnh\cap(0,\infty)\,\mathrm{\ of\ the\ flat\ connection\ in\
}\,\dz.
\end{equation}
Our first main result, established in Section~\ref{pt}, can be stated as
follows.
\begin{theorem}\label{modif}In a generic compact isometric quotient of a
rank-one ECS model manifold, the group\/ $\,K\hskip-1.5pt_+\w$ in\/
{\rm(\ref{ppr})} is not infinite cyclic.
\end{theorem}
The next fact, which we prove at the very end of Section~\ref{co},
holds in a more abstract setting, with no reference to either genericity or
model manifolds.
\begin{theorem}\label{genrl}Given a compact rank-one ECS manifold\/
$\,(M\nh,\g)$, with\/ $\,K\hskip-1.5pt_+\w$ in\/ {\rm(\ref{ppr})} not infinite
cyclic, $\,K\hskip-1.5pt_+\w$ may be trivial, which makes\/ $\,(M\nh,\g)\,$
translational, or else\/ $\,K\hskip-1.5pt_+\w$ is dense in\/ $\,(0,\infty)$,
and then\/ $\,(M\nh,\g)\,$ must be locally homogeneous.
\end{theorem}
The third result trivially follows from Theorems~\ref{modif} 
and~~\ref{genrl}.
\begin{theorem}\label{maith}Every generic compact isometric quotient of a
rank-one ECS model manifold is either translational or locally homogeneous.

In the lo\-cal\-ly-ho\-mo\-ge\-ne\-ous case the group\/ {\rm(\ref{ppr})} is
dense in\/ $\,(0,\infty)$.
\end{theorem}
According to the final clause of our Theorem~\ref{dilat}, compact locally
homogeneous rank-one ECS model manifolds are necessarily
di\-la\-tion\-al. Theorem~\ref{maith} thus has the following consequence.
\begin{corollary}\label{equiv}For a generic compact rank-one ECS manifold
arising as an isometric quotient of a model manifold, the property of 
being di\-la\-tion\-al is equivalent to local homogeneity.
\end{corollary}
Both Theorem~\ref{maith} and Corollary~\ref{equiv} do not really require 
assuming that the manifold is an isometric quotient of a model. Namely, as 
we show in \cite[Corollary D]{derdzinski-terek-ms}, {\it the 
pseu\-\hbox{do\hskip.7pt-}Riem\-ann\-i\-an universal covering of any generic 
compact rank-one ECS manifold is necessarily isometric to one of the model
manifolds}.

Furthermore, according to another result (Theorem E) of the same paper 
\cite{derdzinski-terek-ms}, a generic compact rank-one 
ECS manifold cannot be locally homogeneous. Thus, the final clause of our 
Theorem~\ref{maith} is actually vacuous, and (\ref{all}) follows.
However, Theorem~\ref{maith}, precisely as stated here, 
is a crucial step in the arguments of \cite{derdzinski-terek-ms}.

The paper is organized as follows. 
Sections~\ref{co} and~\ref{ro}, dealing with rank-one ECS manifolds, are
followed by some material from linear algebra and algebraic number theory
(genericity of nil\-po\-tent self-ad\-joint linear en\-do\-mor\-phisms of 
pseu\-\hbox{do\hs-}Euclid\-e\-an spaces, and the cyclic root-group condition 
for $\,\mathrm{GL}\hh(\bbZ)$-pol\-y\-no\-mi\-als), in Sections~\ref{gs}
and~\ref{gp}. Those two are separated by a section devoted to sub\-spaces of 
certain spaces $\,\xe\hs$ of vec\-tor-val\-ued functions on $\,(0,\infty)$,
invariant under an operator $\,\bc\hh T:\xe\nh\to\xe\hs$ which is relevant to
the existence question for generic compact isometric quotients of rank-one
ECS model manifolds. After Section~\ref{ca}, presenting a 
combinatorial argument (Theorem~\ref{noset}) needed to establish 
Theorem~\ref{modif}, comes the final Section~\ref{pt}, where we prove
Theorem~\ref{modif} by contradiction, 
assuming that its hypotheses hold and yet $\,K\hskip-1.5pt_+\w$ in (\ref{ppr})
is infinite cyclic. Lemma~\ref{gnrtr} provides the first important consequence
of this assumption: the existence of a $\,\bc\hh T\nh$-in\-var\-i\-ant vector
sub\-space, of the type discussed in Section~\ref{is}, with the additional
properties (\ref{ace}). Such a sub\-space necessarily satisfies further
conditions, listed in Lemma~\ref{abcde}, and leading -- for reasons stated at
the very end of Section~\ref{pt} --  to a combinatorial structure, the
existence of which contradicts Theorem~\ref{noset}.


\renewcommand{\thetheorem}{\thesection.\arabic{theorem}}
\section{Preliminaries}\label{pr}
\setcounter{equation}{0}
Unless stated otherwise, manifolds and mappings are 
smooth, the former connected. 
The group $\aff(\bbR)\hs$ of af\-fine transformations 
$\hs t\hn\mapsto\nh qt+p\hs$ of $\bbR$, with real $p$ and $\,q\ne0$, has the
in\-dex-two sub\-group $\aff^+\nh(\bbR)\nh=\nh\{(q,p)\in\aff(\bbR):q>0\}$, and
\begin{equation}\label{ntr}
\mathrm{nontrivial\ finite\ sub\-groups\ of\ }\,\aff(\bbR)\,
\mathrm{\ have\ the\ form\ }\,\{(1,0),(-\nnh1,2c)\}
\end{equation}
with any center $\,c\in\bbR\,$ of the reflection $\,(-\nnh1,2c)$. In fact, the
square of any $\,(q,p)\,$ in such a sub\-group $\,\Xi\,$ lies in 
the intersection $\,\Xi\hs\cap\hh\aff^+\nh(\bbR)$, which due to its finiteness
must consist of translations, and hence be trivial.

Every $\,(q,p)\in\aff^+\nh(\bbR)\smallsetminus\{(1,0)\}\,$ is either a
translation ($q=1$), or has a unique fixed point $\,c\,$ (and then we call it 
a {\it di\-la\-tion\/} with center $\,c$, since by choosing $\,c\,$ as the new
origin we turn $\,c\,$ into $\,0\,$ and $\,(q,p)\,$ into $\,(q,0)$). Now,
\begin{equation}\label{abs}
\begin{array}{l}
\mathrm{any\hs\ Abel\-i\-an\hs\ sub\-group\hs\ of\hs\
}\,\hs\aff^+\nh(\bbR)\hs\,\mathrm{\hs\ consists\hs\ of}\\
\mathrm{translations,\ or\ of\ di\-la\-tions\ with\ a\ single\ center,}
\end{array}
\end{equation}
as two commuting self-mappings of a set preserve each other's fix\-ed-point
sets, and so in $\,\aff^+\nh(\bbR)\smallsetminus\{(1,0)\}\,$ two di\-la\-tions
with different centers cannot commute with each other or with a translation.
\begin{lemma}\label{point}Let\/ $\,(\hh\cdot\hs,\hn\cdot\hh)\,$ be a symmetric
bi\-lin\-ear form in a real vector space. If a coset\/ $\,S\,$ of a\/
$\,(\hh\cdot\hs,\hn\cdot\hh)$-null one-di\-men\-sion\-al sub\-space\/ $\,Q\,$ 
is not contained in the $\,(\hh\cdot\hs,\hn\cdot\hh)$-or\-thog\-o\-nal 
complement of\/ $\,Q$, then\/ $\,S\,$ contains a unique\/ 
$\,(\hh\cdot\hs,\hn\cdot\hh)$-null vector.  
\end{lemma}
In fact, $\,S\,$ is parametrized by $\,t\mapsto x=v+tu$, where $\,u\,$ spans
$\,Q\,$ and $\,(v,u)\ne0$, so that $\,(x,x)=(v,v)+2t(v,u)\,$ vanishes for a
unique $\,t\in\bbR$.

Let a group $\,\Gm\hs$ act on a manifold $\,\hm\,$ freely by 
dif\-feo\-mor\-phisms. One calls the action of $\,\Gm\hs$ {\it properly
dis\-con\-tin\-u\-ous\/} if there exists a locally dif\-feo\-mor\-phic
surjective mapping $\,\pi:\hm\to M\,$ onto some manifold $\,M$ such that the 
$\,\pi$-pre\-im\-ages of points of $\,M\,$ coincide with the orbits of the 
$\,\Gm\hs$ action. One
then refers to $\,M\,$ as the {\it quotient\/} of $\,\hm\,$ under the action
of $\,\Gm\hs$ and writes $\,M=\hm\nnh/\hh\Gm$.

For $\,\pi,\hm\nh,M,\Gm\hs$ as above and a flat linear connection
$\,\nabla\hs$ in a vector bundle $\,\vb$ over $\,M\nh$, let $\,\zb\,$ and
$\,\hna\hs$ be the $\,\pi$-pull\-backs of $\,\vb,\nabla\hs$ to $\,\hm\nh$. If
$\,\hm\,$ is also simply connected, the vector space $\,\fv\,$ of all
$\,\hna\nh$-par\-al\-lel sections of $\,\zb\,$ trivializes $\,\zb$, and a
homo\-morphism $\,\Gm\to\mathrm{GL}\hh(\fv)$, known as
the {\it holonomy representation\/} of $\,\nabla\nh$, assigns to
$\,\gamma\in\Gm\hs$ the composite iso\-morphism
\begin{equation}\label{cpi}
\fv\,\to\,\zb\hn_y\w\hs\to\,\vb\hn_x\w\hs\to\,\zb\hn_{\gamma(y)}\w\hs
\to\,\fv\hh,
\end{equation}
described with the aid of any given $\,y\in\hm\,$ and $\,x=\pi(y)$, where the
two middle
arrows denote the identity auto\-mor\-phism of 
$\,\zb\hn_y\w\nh=\vb\hn_x\w\nh=\zb\hn_{\gamma(y)}\w$, and the
first/last one is the evaluation operator or its inverse. Note that
(\ref{cpi}) does not depend on the choice of $\,y\in\hm\nh$, being locally
(and hence globally) constant as a function of $\,y$. To see this, we choose
connected neighborhoods $\,\,\hu\,$ of $\,y\,$ in $\,\hm\,$ and
$\,\,U\,$ of $\,x=\pi(y)$ in $\,M\,$ such that $\,\vb\,$ restricted to
$\,\,U\,$ is trivialized by the space $\,\fv\nh\nnh_U\w$ of its
$\,\nabla\nh$-par\-al\-lel sections and $\,\pi\,$ maps $\,\,\hu\,$
dif\-feo\-mor\-phi\-cal\-ly onto $\,\,U\nh$. The iso\-mor\-phism
$\,\fv\to\fv\nh\nnh_U\w$ arising as the restriction to $\,\,\hu\,$ followed by
the ``identity'' identification via $\,\pi\,$ then allows us to apply 
(\ref{cpi}) to a fixed section from $\,\fv\nh$, using all $\,y\in\hu\,$ at
once.

When $\,\vb\,$ is a real line bundle, with the multiplicative group
\hbox{$\,\mathrm{GL}\hh(\fv)=\bbR\smallsetminus\{0\}$,}
\begin{equation}\label{hol}
\begin{array}{l}
\mathrm{for\ any\ }\,x\,\in\hs M\nh\mathrm{,\ the\ image\ of\ the\ holonomy\ 
representation}\\
\Gm\nh\to\hs\bbR\nnh\smallsetminus\nnh\{0\}\mathrm{\ coincides\ with\ the\
holonomy\ group\ of\ }\hs\nabla\nh\mathrm{\ at\ }\nh x,
\end{array}
\end{equation}
the latter meaning the group of the $\,\nabla\nh$-par\-al\-lel transports
$\,\vb\hn_x\w\to\vb\hn_x\w$ along all the loops at $\,x$. In fact,
if (\ref{cpi}) assigns to $\,\gamma\in\Gm\hs$ the multiplication by 
$\,q\in\bbR\smallsetminus\{0\}$ and $\,y\in\pi^{-\nnh1}\nh(x)\,$ is fixed,
the $\,\nabla\nh$-par\-al\-lel transport $\,\op\,$ along the $\,\pi$-im\-age
of any curve joining $\,y\,$ to $\,\gamma(y)\,$ in $\,\hm\,$ is 
$\,\fv\gets\zb\hn_y\w\nh\gets\vb\hn_x\w$ followed by
$\,\mathrm{Id}\hn_\fv\w$ followed by
$\,\vb\hn_x\w\gets\zb\hn_{\gamma(y)}\w\nh\gets\fv\hh$, the reversed arrows
representing the inverses of those in (\ref{cpi}). Writing 
$\,\mathrm{Id}\hn_\fv\w$ as $\,q\nh^{-\nnh1}$ times (\ref{cpi}), we get
$\,\op\,$ equal to $\,q\nh^{-\nnh1}$ times the identity of 
$\,\vb\hn_x\w$.
\begin{lemma}\label{flows}Suppose that\/
$\,q\in\bbR\smallsetminus\{1,-\nnh1\}\,$ and a dif\-feo\-mor\-phism\/ 
$\,\gamma\in\mathrm{Diff}\hs\hm$ of a manifold\/ $\,\hm\,$ pushes a 
complete nontrivial vector field\/ $\,w\,$ forward onto\/ $\,qw$. If\/ 
$\,\bbR\ni t\mapsto\phi(t,\,\cdot\,)\in\mathrm{Diff}\hs\hm\,$ denotes the flow 
of\/ $\,w$, while a sub\-group\/ $\,\Gm\subseteq\mathrm{Diff}\hs\hm$ 
contains\/ $\,\gamma\,$ and\/ $\,\phi(t,\,\cdot\,)\,$ for some\/ $\,t\ne0$,
then the action of\/ $\,\Gm$ on\/ $\,\hm\,$ cannot be properly discontinuous.
\end{lemma}
\begin{proof}The $\,k\hh$th iteration $\,\gamma^k$ of $\,\gamma$,
for $\,k\in\bbZ$, pushes $\,w$ forward onto $\,q^kw$, giving 
$\,\gamma^k\nnh\circ\phi(t,\,\cdot\,)=\phi(q^kt,\,\cdot\,)
\circ\gamma^k$ for all $\,t\,$ and all $\,k\in\bbZ$, so that
$\,\phi(q^kt,\,\cdot\,)\in\Gm\hs$ with
our fixed $\,t$. Choosing $\,x\in\hm\,$ such that $\,w\nh_x\w\ne0$, and
setting $\,\yt=\mathrm{sgn}\hh(1-|q|)$, we thus get a sequence
$\,\phi(q\hh^{\yt\hn k}t,x)\,$
with mutually distinct terms when $\,k\,$ is large, tending to $\,x$ as
$\,k\to\infty$, which obviously precludes proper discontinuity.
\end{proof}
The conclusion of Lemma~\ref{flows} remains valid when, instead of
$\,\phi(qt,\,\cdot\,)\in\Gm$ for some $\,t$, one assumes periodicity of
the flow of $\,w$, and replaces the condition
$\,\gamma,\phi(t,\,\cdot\,)\in\Gm\,$ with just $\,\gamma\in\Gm\hs$ (and then
uses $\,t\,$ equal to the period of the flow).
\begin{remark}\label{bndpr}
A sub\-mer\-sion from a compact manifold into a connected manifold is 
a bundle projection, which is the compact case of Ehres\-mann's 
fibration theorem \cite[Corollary 8.5.13]{dundas}.
\end{remark}

\section{Compact rank-one ECS manifolds}\label{co}
\setcounter{equation}{0}
Throughout this section $\,(\hm\nh,\hg)\,$ is the pseu\-do\hs-Riem\-ann\-i\-an
universal covering space of a compact rank-one
ECS manifold $\,(M\nh,\g)\,$ of dimension $\,n\ge4$, defined as in the
Introduction, $\,\dz\,$ stands for the (one-di\-men\-sion\-al, null,
parallel) Ol\-szak distribution on $\,(M\nh,\g)$, and $\,\dzp$ for its
orthogonal complement, while $\,\hdz,\,\hdp$ are 
the analogous distributions on $\,(\hm\nh,\hg)$. Thus,
$\,M\nh=\hm\nnh/\hh\Gm\hs$ 
for a sub\-group $\,\Gm\,$ of the full isom\-e\-try group 
$\,\mathrm{Iso}\hs(\hm\nh,\hg)\,$ iso\-mor\-phic to
the fundamental group of $\,M\nh$, and acting on $\,\hm\,$
freely and properly dis\-con\-tin\-u\-ous\-ly via deck transformations. 
The connection in $\,\hdz\,$ induced by the Le\-vi-Ci\-vi\-ta connection
$\,\hna\,$ of $\,(\hm\nh,\hg)\,$ is always flat
\cite[Sect.\ 9]{derdzinski-terek-tc}. Thus, due to simple connectivity of 
$\,\hm\nh$,
\begin{equation}\label{dsp}
\begin{array}{l}
\hdz\hs\mathrm{\ is\ spanned\ by\ the\ parallel\ gradient}\\ 
\hna\hn t\,\hs\mathrm{\ of\ a\ surjective\ function\ }\,t\nnh:\nnh\hm\nh\to I
\end{array}
\end{equation}
onto an open interval $\,I\nh\subseteq\bbR\,$ (which is the case even
without assuming the existence of a compact quotient).
The Olszak distribution being a local geometric invariant of the ECS metric
in question \cite[Sect.\ 2]{derdzinski-roter-09}, (\ref{dsp}) determines
$\,\hna\hn t\,$ and $\,t\,$ uniquely up to multiplication by
nonzero constants and, respectively, af\-fine substitutions, meaning  
replacements of $\,t\,$ with $\,qt+p$, where $\,(q,p)\in\aff(\bbR)\,$ (for
$\,\aff(\bbR)\,$ as in Section~\ref{pr}: $\,q,p\in\bbR\,$ and $\,q\ne0$).
Consequently, we have group homo\-mor\-phisms
\begin{equation}\label{hom}
\mathrm{a)}\hskip7pt\mathrm{Iso}\hs(\hm\nh,\hg)\ni\gamma\mapsto(q,p)
\in\aff(\bbR)\hh,
\hskip11pt
\mathrm{b)}\hskip7pt\mathrm{Iso}\hs(\hm\nh,\hg)\ni\gamma\mapsto
q\in\bbR\smallsetminus\{0\}\hh,
\end{equation}
characterized, for any $\,\gamma\in\mathrm{Iso}\hs(\hm\nh,\hg)$, 
by $\,t\circ\gamma=qt+p\,$ and $\,\gamma^*\hn dt=q\,dt$, that is,
\begin{equation}\label{dgn}
(d\gamma)\hna\hn t\,=\,q\nh^{-\nnh1}\hna\hn t\hh.
\end{equation}
According to
\cite[formula (5.4) and the end of Sect.\ 11]{derdzinski-terek-tc},
\begin{equation}\label{ker}
\begin{array}{l}
\hdp\nh=\,\hs\mathrm{Ker}\,dt\mathrm{,\ the\ levels\ of\ 
}\,\,t:\hm\nh\to I\,\hs\mathrm{\ are\ all}\\
\mathrm{connected\ and\ coincide\ with\ the\ leaves\ of\ }\,\hdp\nnh.
\end{array}
\end{equation}
\begin{lemma}\label{image}The above hypotheses imply that the image of\/
$\,\Gm\hs$ under\/ {\rm(\ref{hom}-a)} is infinite, while its 
image under\/ {\rm(\ref{hom}-b)} coincides with the holonomy group of the
flat connection in\/ $\,\dz$.
\end{lemma}
\begin{proof}The first image, if finite, would, lie within some 
$\,\{(1,0),(-\nnh1,2c)\}$, cf.\ (\ref{ntr}), causing $\,(t-c)^2$ to
descend to a nonconstant function with at most one critical value
on the compact manifold $\,M\nh$. The second claim follows from (\ref{hol}): 
by (\ref{dsp}) and (\ref{dgn}), the action (\ref{cpi}) of any
$\,\gamma\in\Gm\hs$ on the parallel section $\,\hna\hn t\,$ spanning
$\,\hdz\,$ equals the multiplication by the corresponding $\,q\nh^{-\nnh1}\nh$.
Namely, the two middle arrows in (\ref{cpi}) now are restrictions of 
$\,d\pi_y\w$ and $\,[d\pi_{\gamma(y)}\w]\nh^{-\nnh1}\nh$, so that their
composite $\,\zb\hn_y\w\nh\to\vb\hn_x\w\nh\to\zb\hn_{\gamma(y)}\w$ equals
$\,d\gamma\nh_y\w$. (From $\,\pi\circ\gamma=\pi\,$ we get
$\,d\pi_{\gamma(y)}\w\circ\,d\gamma\nh_y\w=\,d\pi_y\w$.) Thus, (\ref{cpi})
takes
$\,w=\hna\hn t\,$ first to $\,w\hn_y\w$, then (two successive arrows) to
$\,d\gamma\nh_y\w w\hn_y\w$ which -- by (\ref{dgn}) -- equals
$\,q\nh^{-\nnh1}\nh w\hn_{\gamma(y)}\w$, the evaluation at $\,\gamma(y)\,$ of 
$\,q\nh^{-\nnh1}\nh w$.
\end{proof}
The trans\-la\-tion\-al/\hn di\-la\-tion\-al dichotomy of the Introduction,
meaning fi\-ni\-te\-ness/in\-fi\-ni\-te\-ness of the holonomy group of the
flat connection in $\,\dz\,$ induced by the Le\-vi-Ci\-vi\-ta connection of
$\,\g$, can now be summarized in terms of the homo\-mor\-phism
(\ref{hom}-b) restricted to $\,\Gm\nh$. Specifically, by 
Lemma~\ref{image}, the two cases are
\begin{equation}\label{dch}
\begin{array}{rll}
\mathrm{a)}\hskip4pt&\mathrm{trans\-la\-tion\-al\hskip-3pt:\ }&\,|q|
=1\,\mathrm{\ for\ each\ }\,\gamma\in\Gm,\\
\mathrm{b)}\hskip4pt&\mathrm{di\-la\-tion\-al\hskip-3pt:\
}&\,|q|\ne1\,\mathrm{\ for\ some\ }\,\gamma\in\Gm.
\end{array}
\end{equation}
\begin{lemma}\label{cplfl}With the assumptions and notations as above,
\begin{enumerate}
\item[(a)] the parallel vector field $\,\hna\hn t\,$ on $\,\hm\nh$, spanning
$\,\hdz$, is complete,
\item[(b)] in case\/ \/ {\rm(\ref{dch}-b)}, $\,\phi(t,\,\cdot\,)\notin\Gm\hs$
for all\/ $\,t\in\bbR\smallsetminus\{0\}$,
\end{enumerate}
$\bbR\ni t\mapsto\phi(t,\,\cdot\,)\in\mathrm{Diff}\hs\hm\,$ being the flow
of\/ $\hna\hn t$.
\end{lemma}
In fact, (a) appears in
\cite[the second italicized conclusion in Sect.\ 11]{derdzinski-terek-tc},
while (b) follows from (a) and Lemma~\ref{flows} combined with (\ref{dgn}).

The remainder of this section uses the assumptions preceding (\ref{dsp})
along with
\begin{equation}\label{isp}
\mathrm{transversal\ orientability\ of\ }\,\dzp\nnh\mathrm{\ which,\ by\ 
(\ref{ker}),\ reads\ }\,\Gm\,\subseteq\,\mathrm{Iso}\nh^+\nh(\hm\nh,\hg)\hh,
\end{equation}
for the normal sub\-group $\,\mathrm{Iso}\nh^+\nh(\hm\nh,\hg)\,$ forming the
(\ref{hom}-b)-pre\-im\-age of $\,(0,\infty)$. This can always be achieved by
replacing $\,(M\nh,\g)\,$ (or, $\,\Gm$) with a two-fold isometric covering
(or, an in\-dex-two sub\-group), and has an obvious consequence: the 
translational case then  means precisely that the holonomy group is trivial.
\begin{theorem}\label{dilat}In the di\-la\-tion\-al case\/ {\rm(\ref{dch}-b)},
with\/ {\rm(\ref{isp})}, the image of\/ $\,\Gm\hs$ under\/ {\rm(\ref{hom}-a)} 
consists of di\-la\-tions with a single center. The replacement of\/ $\,t\,$
in\/ {\rm(\ref{dsp})} by a suitable af\-fine function of\/ $\,t\,$ then makes 
this center appear as\/ $\,t=0$, the interval\/ $\,I\hh$ as\/ $\,(0,\infty)$,
and all\/ $\,(q,p)\,$ in the\/ {\rm(\ref{hom}-a)}-im\-age of\/
$\,\Gm\hs$ as having\/ $\,p=0$.

Then the image of\/ $\,\Gm\hs$ under\/ {\rm(\ref{hom}-a)}, always an infinite
multiplicative sub\-group of\/ $\,(0,\infty)$, must be infinite cyclic unless\/
$\,(\hm\nh,\hg)\,$ is locally homogeneous.

On the other hand, {\rm(\ref{dch}-b)} follows if one assumes local homogeneity
of\/ $\,(\hm\nh,\hg)$.
\end{theorem}
\begin{proof}As shown in
\cite[the beginning of Sect.\ 11]{derdzinski-terek-tc}, (\ref{isp}) implies
the existence of a $\,C^\infty\nnh$ function $\,\psi:\hm\to(0,\infty)\,$ such
that the $\,1$-form $\,\psi\,dt\,$ is $\,\pi$-pro\-ject\-a\-ble onto $\,M\,$
(in other words, $\,\Gm\nh$-in\-var\-i\-ant), and closed. According to
(\ref{ker}), the $\,t$-levels in $\,\hm\,$ are all connected, and so
closedness of $\,\psi\,dt\,$ makes $\,t\,$ {\it globally\/} a function of
$\,t$, with $\,\psi=\chi\circ t\,$ for some $\,C^\infty\nnh$ function
$\,\chi:I\nh\to(0,\infty)$. A fixed anti\-de\-riv\-a\-tive $\,\phi\,$ of
$\,\chi\,$ thus constitutes a strictly increasing $\,C^\infty\nnh$
dif\-feo\-mor\-phism $\,\phi:I\nh\to J\,$ onto some open interval
$\,J\subseteq\bbR$, while $\,\Gm\nh$-in\-var\-i\-ance of
$\,d\hs(\phi\circ t)=\psi\,dt\,$ means that $\,\Gm\hs$ acts on
$\,\phi\circ t\,$ by translations:
$\,\phi\circ t\circ\gamma=\phi\circ t+a\,$ with constants $\,a\in\bbR\,$
depending on $\,\gamma\in\Gm\nh$. The mappings $\,t:\hm\nh\to I\hs$ and
$\,\phi\circ t:\hm\nh\to J\,$ are $\Gm\nh$-equi\-var\-i\-ant relative to
$\,\Gm$ acting on $\,I\hs$ and $\,J\,$ via the homo\-mor\-phisms
(\ref{hom}-a), restricted to $\,\Gm\nh$, and $\,\gamma\mapsto a$. As the 
dif\-feo\-mor\-phism $\,\phi:I\nh\to J\,$ makes the two mappings
equi\-var\-i\-ant\-ly equivalent, the two homo\-mor\-phisms have the same
kernel $\,\Sigma\subseteq\Gm\nh$, leading to an iso\-mor\-phism
$\,(q,p)\mapsto\gamma\Sigma\mapsto a\,$ between the images of the two
homo\-mor\-phisms. The former image must thus be Abel\-i\-an (as that of
$\,\gamma\mapsto a\,$ is a group of translations) and so, due to (\ref{dch}-b) and (\ref{abs}), it 
consists of di\-la\-tions with a single center. An af\-fine substitution
of $\,t\,$ turns this center into $\,0$, and elements of the
(\ref{hom}-a)-im\-age of $\,\Gm$ into pairs $\,(q,p)\,$ with $\,q>0\,$
and $\,p=0$. As a result, for our open interval $\,I\nh$,
\begin{enumerate}
\item[(i)] $0\,$ lies in the closure of $\,I\nh$, but not in $\,I\hs$ itself.
\end{enumerate}
The first claim of (i) is obvious: by (\ref{dch}-b) -- (\ref{isp}), 
for some $\,q\in(0,\infty)\smallsetminus\{1\}$,
\begin{enumerate}
\item[(ii)] $I\,$ is closed under multiplications by powers of $\hs q\hs$
with integer exponents.
\end{enumerate}
To verify the second one, note that, as shown in
\cite[formulae (5.6) -- (5.8)]{derdzinski-terek-tc}, some nonconstant
$\,C^\infty\nnh$ function $\,f:\hm\to\bbR\,$ has
\begin{enumerate}
\item[(iii)] $f\circ\gamma=q^{-\nh2}f\,$ for all $\,\gamma\in\Gm\hs$ and 
$\,(q,p)\in\aff^+\nh(\bbR)\,$ with $\,t\circ\gamma=qt+p$. 
\end{enumerate}
This $\,f\,$ is also globally a function of $\,t\,$ 
\cite[the end of Sect.\ 11]{derdzinski-terek-tc}. Treating $\,f\nh$,
informally, as a function $\,I\nh\to\bbR$, and noting that all 
$\,(q,p)\,$ in the (\ref{hom}-a)-im\-age of $\,\Gm$ now have $\,q>0\,$ and
$\,p=0$, we get $\,f(t)=\hs q^2\nnh f(qt)\,$ for such $\,q$, 
while these $\,q$, due to Lemma~\ref{image}, form an infinite sub\-group of
$\,(0,\infty)$. Thus, 
$\,0\notin I\nh$, or else, fixing any $\,t\,$ in the equality 
$\,f(t)=\hs q^2\nnh f(qt)\,$ and letting
$\,q\to0$, we would get $\,f(t)=0$, 
even though $\,f$ is nonconstant.

By (i) and (ii), $\,I\hs$ equals $\,(0,\infty)\,$ or $\,(-\infty,0)\,$
and, replacing $\,t\,$ with $\,-t\,$ if necessary, we get
$\,I\nh=(0,\infty)$, proving the first assertion of the theorem.

To establish the second one, suppose that the (\ref{hom}-a)-im\-age of 
$\,\Gm\nh$, infinite as a consequence of Lemma~\ref{image}, is not cyclic. This
makes the image dense in $\,(0,\infty)$, so that, from continuity of
$\,f\nh$, our equation $\,f(t)=\hs q^2\nnh f(qt)\,$ holds for {\it all\/} 
$\,t,q\in(0,\infty)$. Setting $\,t=1$, we get $\,f(q)=f(1)/q^2$.
The resulting linearity of the function $\,|f|^{-\nnh1/2}$
amounts -- see \cite[Theorem 6.3]{derdzinski-terek-tc} -- to local
homogeneity of $\,(\hm\nh,\hg)$.

Finally, suppose that $\,(\hm\nh,\hg)\,$ is locally homogeneous. The preceding
lines now yield linearity of
$\,|f|^{-\nnh1/2}\nnh$, that is, $\,f(t)=f(1)/t^2\,$ for all 
$\,t\in(0,\infty)$, and so $\,f$ is unbounded on $\,(0,\infty)$. This 
gives (\ref{dch}-b), since (\ref{dch}-a) would, by (iii), imply 
$\,\Gm\nh$-in\-var\-i\-ance of $\,f\nh$, leading to its boundedness, as 
$\,M\nh=\hm\nnh/\hh\Gm\hs$ is compact.
\end{proof}
\begin{proof}[Proof of Theorem~\ref{genrl}]Due to Lemma~\ref{image} we may,
without loss of generality, assume (\ref{dch}-b) and (\ref{isp}). Our claim
now follows from Theorem~\ref{dilat}.
\end{proof}

\section{The rank-one ECS model manifolds}\label{ro}
\setcounter{equation}{0}
In this section we fix the data $\,f\nh,I\nh,n,\mv\nh,\lr,A\,$ consisting of
\begin{equation}\label{dta}
\begin{array}{l}
\mathrm{an\ integer\ }\,\,n\,\ge\,4\mathrm{,\ a\ real\ vector\ space\
}\,\,\mv\hh\mathrm{\ of\ dimension\ }\,\,n\hs-\hs2,\\
\mathrm{a\ pseu\-do\hs}\hyp\mathrm{Euclid\-e\-an\nh\
inner\ product\ }\nh\lr\hs\mathrm{\ on\ }\hs\mv\hskip-3pt\mathrm{,\nh\ 
a\ nonzero,\ trace}\hyp\\
\mathrm{less,\ }\hn\lr\hyp\mathrm{self}\hyp\mathrm{ad\-joint\ linear\ 
en\-do\-mor\-phism\ }\hs A\hs\mathrm{\ of\ }\,\hs\mv\nnh\mathrm{,and\ a\
non}\hyp\\
\mathrm{constant\ }\,\,C^\infty\hs\mathrm{\ function\ 
}\,\,f:I\to\hs\bbR\,\,\mathrm{\ on\ an\ open\ interval\ 
}\,\,I\subseteq\hn\bbR.
\end{array}
\end{equation}
Treating $\,\lr\,$ as a flat (constant) metric on $\,\mv\nh$, and following 
\cite{roter}, we define the simply connected $\,n$-di\-men\-sion\-al
pseu\-do\hs-Riem\-ann\-i\-an manifold
\begin{equation}\label{met}
(\hm\nh,\hg)\,
=\,(I\nh\times\bbR\times\mv\nh,\,\kappa\,dt^2\nh+\,dt\,ds\hs+\hs\lr)\hh,
\end{equation}
where $\,t,s\,$ are the Cartesian coordinates on $\,I\nh\times\bbR$, we
identify $\,dt,\,ds\,$ and $\,\lr$ with their pull\-backs to
$\,\hm\nh$, and the function $\,\kappa:\hm\nh\to\bbR\,$ is defined by 
$\,\kappa(t,s,v)=f(t)\hskip.4pt\langle v,v\rangle+\langle Av,v\rangle$. Thus,
translations in the $\hs s\hs$ direction are isometries of
$\,(\hm\nh,\hg)$.

It is well known \cite[Theorem~4.1]{derdzinski-roter-09} that (\ref{met}) is a
rank-one ECS manifold. To describe its isometry group, we need two
ingredients. The first is
\begin{equation}\label{frs}
\begin{array}{l}
\mathrm{the\ subgroup\ }\hs\mathrm{S}\hs\mathrm{\ of
}\,\aff(\bbR)\times\mathrm{O}\hh(\mv)\,\mathrm{\ formed\ by}\\
\mathrm{triples\ }\,(q,p,\bc)\,\mathrm{\ such\ that\ 
}\hs\bc\hskip-1.8ptA\bc^{-\nnh1}\nnh=\hs q^2\hskip-2.3ptA\mathrm{, while}\\
qt+p\in\hn I\hs\mathrm{\ and\ }\hs f(t)=\hs q^2\nnh f(qt+p)\hs\mathrm{\ for\
all\ }\hs t\in I\nh,
\end{array}
\end{equation}
$\mathrm{O}\hh(\mv)\,$ being the group of linear $\,\lr$-isom\-e\-tries
$\,\bc:\mv\to\mv\nh$.

The second ingredient is the $\,2(n-2)$-di\-men\-sion\-al real
\begin{equation}\label{vss}
\begin{array}{l}
\mathrm{vector\ space\ }\,\xe\hs\mathrm{\ of\ all\ solutions\ 
}\,u:I\nh\to\mv\hs\mathrm{\ to\ the\ sec\-ond}\hyp\mathrm{or\-der\ or}\hyp\\
\mathrm{dinary\ differential\ equation\ }\hs\ddot u
\hn=\nh f\nh u+\nnh Au\mathrm{,\nnh\ carrying\ the\ symplectic}\\
\mathrm{form\ }\,\varOmega:\xe\hn\times\hh\xe\nh\to\bbR\,\mathrm{\ 
given\ by\ }\,\,\varOmega(u\nh^+\nnh,u^-)\,
=\,\langle\dot u\nh^+\nnh,u^-\rangle\,-\,\langle u\nh^+\nnh,\dot u^-\rangle\hn.
\end{array}
\end{equation}
Note that $\,q,(q,p),\bc\,$ all depend homo\-mor\-phic\-ally
on the triple $\,\sigma=(q,p,\bc)$, and $\,\mathrm{S}\,$ acts from the left on 
$\,C^\infty\nnh(I\nh,\mv)\,$ via
\begin{equation}\label{acs}
[\hs\sigma\hn u](t)\,=\,\bc u((t-p)/q)\hh,
\end{equation}
while the operator $\,u\mapsto\sigma\hn u\,$ leaves the solution space
$\,\xe\hs$ invariant.
\begin{theorem}\label{isogp}For\/ $\,(\hm\nh,\hg)\,$ and\/
$\,\mathrm{S}\,$ 
as in\/ {\rm(\ref{dta})} -- {\rm(\ref{frs})}, the full isom\-e\-try group 
$\,\mathrm{Iso}\hs(\hm\nh,\hg)\,$ is iso\-mor\-phic to the set\/
$\,\gp=\mathrm{S}\times\bbR\times\xe\nh
\subseteq\aff(\bbR)\times\mathrm{O}\hh(V)\times\bbR\times\xe$ endowed 
with the group operation 
\begin{equation}\label{mlt}
\begin{array}{l}
(q,p,\bc\nh,r,u)(\hat q,\hat p,\hat\bc\nh,\hat r,\hat u)\\
\hskip20pt=\hs(q\hat q,\hs q\hat p+p,\hs 
\bc\hat \bc\nh,\hs-\varOmega(u,(q,p,\bc)\hat u)+r
+\hat r/q,\hs(q,p,\bc)\hat u+u)\hh\\
\mathrm{or,\ in\ the\ notation\ of\ (\ref{vss})\ -\ (\ref{acs}),\ with\ }
\,\sigma=(q,p,\bc),\\
(\sigma,r,u)(\hat\sigma,\hat r,\hat u)
=(\sigma\hat\sigma,\hs\varOmega(\sigma\hat u,u)+r+\hat r/q,
\hs\sigma\hat u+u)\hh.
\end{array}
\end{equation}
The required isomorphism amounts to the following left action on\/ $\,\hm\,$
by the group\/ $\,\gp\,$ with the operation\/ {\rm(\ref{mlt}):}
\begin{equation}\label{act}
\begin{array}{l}
(q,p,\bc\nh,r,u)(t,s,v)\\
=\hs(qt\nh+\nh p,\hh-\langle\dot u(qt\nh+\nh p),\hh 2\hh\bc v
+u(qt\nh+\nh p)\rangle+r+s/q,\hh \bc v+u(qt\nh+\nh p))\hh.
\end{array}
\end{equation}
\end{theorem}
\begin{proof}This is precisely \cite[Theorem 2]{derdzinski-80}, plus 
\cite[p.\ 24, formula (22)]{derdzinski-80} describing the group operation, 
except for 
the fact that \cite{derdzinski-80} assumes real-an\-a\-lyt\-ic\-i\-ty of
$\,f$ along with $\,I\nh=\bbR$, and it is because of these assumptions that
$\,|q|=1\,$ whenever $\,(q,p,\bc)\in\mathrm{S}$, cf.\ (\ref{frs}). 
If one ignores the last conclusion and the assumptions that led to it,
the proof in \cite{derdzinski-80} repeated almost verbatim in our case yields
our assertion. However, the resulting right-hand side in (\ref{act}) is not
ours, but instead reads
\[
(qt\nh+\nh p,\hh-\langle\dot u(t),\hh 2\hh\bc v
+u(t)\rangle+r+s/q,\hh \bc v+u(t))
\]
due to the fact that $\,u$, instead of $\,\xe\nh$, 
now lies in the solution space $\,\xe\nnh_q\w$ of the $\,q$-de\-pend\-ent
equation 
$\,\ddot u=fu+q^2\hskip-2.3ptAu$. We reconcile both versions by observing that
the replacement of $\,u\,$ with $\,t\mapsto u(qt+p)\,$ defines an
iso\-mor\-phism $\,\xe\nnh_q\w\to\xe\nh$.

The notation of \cite{derdzinski-80} differs from ours: our 
$\,q,p,\bc\nh,r,u,t,s,v,\mv\nh,f\nh,\kappa,A,\lr,\varOmega$ correspond to 
$\,\ve,T\nh,H^\lambda_\mu,r,C^\lambda\nnh,x^1\nh,2x^n\nh,
\bbR^{\hskip-.6ptn-2}\nnh,A,\varphi,a_{\lambda\mu}\w,k_{\lambda\mu}\w,
2\hs\omega\,$ in \cite{derdzinski-80}.
\end{proof}
By (\ref{mlt}), $\,\gp\ni\gamma=(\sigma,r,u)\mapsto\sigma\in\mathrm{S}\,$ is
a group homo\-mor\-phism, leading to
\begin{equation}\label{ghm}
\mathrm{the\ normal\ sub\-group\ }\,\,\hp\,
=\,\{(1,0,\mathrm{Id})\}\times\bbR\times\xe\,\hs\mathrm{\ of\ }\,\,\gp\hh.
\end{equation}
The group operation (\ref{mlt}) restricted to $\,\hp\,$ becomes
\begin{enumerate}
\item[(a)] $(1,0,\mathrm{Id},\hat r,\hat u)(1,0,\mathrm{Id},r,u)
=(1,\,0,\,\mathrm{Id},\,\varOmega(u,\hat u)+\hat r+r,\,\hat u+u)$,
\end{enumerate}
and the action (\ref{act}) of $\,\hp\,$ on $\,\hm\,$ is explicitly given by
\begin{enumerate}
\item[(b)] $(1,0,\mathrm{Id},r,u)(t,s,v)
=(t,\,\hh-\hn\langle\dot u(t),\hh 2v+u(t)\rangle+r+s,\,\hh v+u(t))$.
\end{enumerate}
Treating the vector space $\,\xe\hs$ as an Abel\-i\-an group 
we get, from (a), an obvious
\begin{enumerate}
\item[(c)] group homo\-mor\-phism
$\,\hp\ni(1,0,\mathrm{Id},r,u)\mapsto u\in\xe\nh$.
\end{enumerate}
Also, as stated in \cite[formula (5.5)]{derdzinski-terek-tc}, with a suitable
af\-fine substitution,
\begin{enumerate}
\item[(d)] $t\,$ in (\ref{met}) can always be made equal to $\,t\,$ chosen as
in (\ref{dsp}),
\end{enumerate}
so that, in view of (\ref{mlt}) -- (\ref{act}),
\begin{enumerate}
\item[(e)] the homo\-mor\-phism $\,\gp\ni(q,p,\bc\nh,r,u)\mapsto(q,p)\,$
coincides with (\ref{hom}-a).
\end{enumerate}
Furthermore, according to
\cite[the lines following formula (3.6)]{derdzinski-terek-ne},
$\,\hna\hn t\,$ in (\ref{dsp}) equals twice the coordinate
vector field in the $\,s\,$ coordinate direction, and so
\begin{enumerate}
\item[(f)] the flow of $\,\hna\hn t\,$ on $\,\hm\,$ is given by 
$\,\bbR\ni r\mapsto(1,0,\mathrm{Id},2r,0)\in\hp\subseteq\gp$.
\end{enumerate}
In other words, cf.\ (b), the flow acts on $\,\hm\,$ via
$\,(\vp,(t,s,v))\mapsto(t,s+2\vp,v)$. Also,
\begin{enumerate}
\item[(g)] $\sigma\hn^*\nnh\varOmega=q\nh^{-\nnh1}\nnh\varOmega$, as an
  obvious consequence of (\ref{vss}) -- (\ref{acs}).
\end{enumerate}
The sub\-group $\,\hp\,$ (canonically identified
with $\,\bbR\times\xe$) acts both on the product
$\,I\nh\times\bbR\times\xe\nh$, by left $\,\hp$-trans\-la\-tions of 
the $\,\hp\,\approx\,\bbR\times\xe\,$ factor, and on $\,\hm\nh$, via
(b). The following mapping is $\,\hp$-equi\-var\-i\-ant for these two
actions:
\begin{equation}\label{prj}
I\nh\times\bbR\times\xe\ni(t,z,u)\,\mapsto(t,s,v)\,
=\,(t,z\,-\,\langle\dot u(t),u(t)\rangle,\,u(t))\in\hm\nh
=I\nh\times\bbR\times\mv\nh.
\end{equation}
as one easily verifies using (a), (b) and the definition of $\,\varOmega\,$ in 
(\ref{vss}).
\begin{remark}\label{cnjug}It is useful to note that
$\,(\sigma,r,u)\nh^{-\nnh1}\nh
=(\sigma\nh^{-\nnh1}\nh,-qr,-\sigma\nh^{-\nnh1}\nh u)\,$
in $\,\gp$, which yields, for
$\,(\sigma,\hat r,\hat u)=(q,p,\bc,\hat r,\hat u)\in\gp\,$ 
and $\,(1,0,\mathrm{Id},r,u)\in\hp$, the equality
\[
(\sigma,\hat r,\hat u)(1,0,\mathrm{Id},r,u)
(\sigma,\hat r,\hat u)^{-\nnh1}\nnh
=\hs(1,0,\mathrm{Id},\,2\hh\varOmega(\sigma\hn u,\hat u)+r/q,
\,\sigma\hn u)\hh.
\]
\end{remark}
\begin{remark}\label{lagrg}Nondegeneracy of $\,\varOmega\,$ gives 
$\,\dim\lz'\nnh=\dim\hh\xe-\hh\dim\lz\,$ for any vector sub\-space 
$\,\lz\subset\xe\hs$ and its $\,\varOmega$-or\-thog\-o\-nal complement 
$\,\lz'\nh$. Thus, $\,2\hs\dim\lz\le\dim\hh\xe$ whenever $\,\lz\,$\ is
isotropic in the sense that $\,\varOmega(u,u')=0\,$ for all 
$\,u,u'\in\lz$.
\end{remark}
\begin{remark}\label{gnric}We refer to a rank-one ECS model manifold
(\ref{met}) as {\it generic\/} when so is $\,A\,$ in (\ref{dta}), by which we 
mean that $\,A\,$ commutes with only finitely many linear
$\,\lr$-iso\-metries of $\,\mv\nh$. Genericity of $\,A\,$ in (\ref{dta}) is
an intrinsic property of the metric $\,\hg$, rather than just a condition
imposed on the construction (\ref{dta}) -- (\ref{met}): as stated in 
\cite[the paragraph following formula (6.3)]{derdzinski-terek-tc}, the
algebraic type of the pair $\,\lr,A$, up to rescaling of $\,A$, can be
explicitly defined in terms of $\,\hg\,$ and its Weyl tensor.
\end{remark}
\begin{remark}\label{impnl}The relation
$\,\bc\hskip-1.8ptA\bc^{-\nnh1}\nnh=\hs q^2\hskip-2.3ptA\,$ in (\ref{frs}) 
with $\,|q|\ne1\,$ implies nil\-po\-tency of $\,A$, as all complex 
characteristic roots of $\,A\,$ then obviously equal $\,0$.
\end{remark}

\section{Generic self-ad\-joint nil\-po\-tent en\-do\-mor\-phisms}\label{gs}
\setcounter{equation}{0}
Throughout this section $\,\mv\hs$ denotes a real vector space of dimension
$\,m\ge2$.

Given a pseu\-\hbox{do\hs-}Euclid\-e\-an inner product $\,\lr\,$ on
$\,\mv\nh$, we refer to $\,\lr\,$ as {\it sem\-i-neu\-tral\/} if its positive
and negative indices differ by at most one, and -- following the terminology
of Remark~\ref{gnric} -- call a $\,\lr$-self-ad\-joint 
en\-do\-mor\-phism of $\,\mv\hs$ {\it generic\/} when it commutes with
only a finite number of linear $\,\lr$-iso\-metries of $\,\mv\nh$. As we show
below (Remark~\ref{nlpgn}), for $\,\lr$-self-ad\-joint en\-do\-mor\-phisms
$\,A\,$ of $\,\mv\hs$ which are nil\-po\-tent, genericity is equivalent to
having $\,A\nh^{m-1}\nnh\ne0\,$ (while $\,A\nh^m\nh=0$).

Nil\-po\-tent en\-do\-mor\-phisms are relevant to our discussion due to
Remark~\ref{impnl}.

Generally, for any en\-do\-mor\-phism $\,A\,$ of our vector space $\,\mv\hs$
and any integer $\,j\ge1$, the inclusions
$\,\mathrm{Ker}\hskip1.7ptA\hn^{j-1}\subseteq\mathrm{Ker}\hskip1.7ptA\hn^j$
allow us to define the quotient spaces
$\,\mathrm{Ker}\hskip1.7ptA\hn^j\nnh\nh
/\mathrm{Ker}\hskip1.7ptA\hn^{j-1}\nnh$, and then $\,A\,$ obviously descends
to {\it injective linear operators}
\begin{equation}\label{inj}
A:\mathrm{Ker}\hskip1.7ptA\hn^{j+1}\nnh\nh/\mathrm{Ker}\hskip1.7ptA\hn^j\hs
\to\hs\mathrm{Ker}\hskip1.7ptA\hn^j\nnh\nh/\mathrm{Ker}\hskip1.7ptA\hn^{j-1},
\hskip9ptj=1,\dots,m-1\hh.
\end{equation}
Setting $\,d\hn_j\w=\dim\,[\mathrm{Ker}\hskip1.7ptA\hn^j\nnh\nh
/\mathrm{Ker}\hskip1.7ptA\hn^{j-1}]\,$ we thus have
$\,d\hn_j\w\ge d\hn_{j+1}\w$ and, if $\,A\,$ is nil\-po\-tent,
\begin{equation}\label{dog}
d_1\w\ge\ldots\ge d_m\w\ge\,0\hskip6pt\mathrm{and\ }\,\dim\mv\nh
=\,d_1\w+\ldots+d_m\w\hh,
\end{equation}
while, whenever $\,j=0,\dots,m$,
\begin{equation}\label{dmk}
\dim\,\mathrm{Ker}\hskip1.7ptA\hn^j\nh=d_1\w+\ldots+d\hn_j\w\hh,\hskip12pt
\mathrm{rank}\hskip1.7ptA\hn^j\nh=\,d\hn_{j+1}\w+\ldots+d_m\w\hh.
\end{equation}
Thus, $\,d_m\w\ge1\,$ in the case where $\,A\,$ is nil\-po\-tent and
$\,A\nh^{m-1}\nnh\ne0$, and then, by (\ref{dog}),
\begin{equation}\label{deo}
d_1\w=\ldots=\hs d_m\w=1\hskip6pt\mathrm{and\ (\ref{inj})\ is\ an\
isomorphism\ for\ }\,j=1,\dots,m-1\hh.
\end{equation}
\begin{theorem}\label{gnnlp}Let a\/ $\,\lr$-self-ad\-joint nil\-po\-tent
en\-do\-mor\-phism\/ $\,A\,$ of an\/ $\,m$-di\-men\-sion\-al
pseu\-\hbox{do\hs-}Euclid\-e\-an vector 
space\/ $\,\mv\hs$ have\/ $\,A\nh^{m-1}\nnh\ne0$. Then the inner product\/ 
$\,\lr\,$ is\/ sem\-i-neu\-tral and there 
exist exactly two bases\/ $\,e\nh_1\w,\dots,e\nh_m\w$ of\/ $\,\mv\nnh$,
differing by an overall sign change, as well as a unique sign factor\/ 
$\,\ve=\pm1$, such that\/ $\,Ae\nnh_j\w=e\nh_{j-1}\w$ and\/ 
$\,\langle e\nh_i\w,e\nh_k\w\rangle=\ve\hn\delta\nh_{i\hn j}\w$ for all\/ 
$\,i,j\in\{1,\dots,m\}$, where\/ $\,e\nh_0\w=0\,$ and\/ $\,k=m+1-j$.
Equivalently, the matrix representing $\,A\,$ or, respectively, $\,\lr\,$ in
our basis has zero entries except those immediately above main diagonal, all
equal to\/ $\,1\,$ or, respectively, except those on the main
anti\-di\-ag\-o\-nal, all equal to\/ $\,\ve$.

Conversely, if\/ $\,A\,$ and\/ $\,\lr\,$ are of the above form in some basis\/
$\,e\nh_1\w,\dots,e\nh_m\w$ of\/ $\,\mv\nnh$, then\/ $\,A\,$ is\/
$\,\lr$-self-ad\-joint, nil\-po\-tent and\/ $\,A\nh^{m-1}\nnh\ne0$.
\end{theorem}
\begin{proof}For $\,j=0,\dots,m$, the symmetric bi\-lin\-ear form
$\,(v,v')\mapsto\langle A\hn^j\nh v,v'\rangle\,$ on $\,\mv\nnh$, briefly denoted by 
$\,\langle A\hn^j\hn\cdot\hs,\hn\cdot\hh\rangle$, and the sub\-spaces
$\,\mv\hskip-3.6pt_j\w=A\hn^j\nh(V)\subseteq\mv\nnh$,
\begin{enumerate}
\item[(a)] $\dim\mv\hskip-3.6pt_j\w=m-j\,$ and
$\,\mv\hskip-3.6pt_j\w\subseteq\mv\hskip-3.6pt_{j-1}\w$ if $\,j\ge1$,
\item[(b)] $\langle A\hn^{m-j}\hn\cdot\hs,\hn\cdot\hh\rangle\,$ descends to
the $\,j$-di\-men\-sion\-al quotient space
$\,\mv\hskip-2.4pt/\mv\hskip-3.6pt_j\w$,
\end{enumerate}
(a) being obvious from (\ref{dmk}) -- (\ref{deo}), and (b) from
\begin{enumerate}
\item [(c)] self-ad\-joint\-ness of $\,A\,$ along with the relation
$\,A\nh^m\nh=0$.
\end{enumerate}
As $\,A\nh^{m-1}\nnh\ne0$, the form resulting from (b) on the line 
$\,\mv\hskip-2.4pt/\mv\hskip-3.5pt_1\w$ is nonzero, and hence positive or
negative definite, which proves the existence and uniqueness of a sign factor 
$\,\ve\in\{1,-\nnh1\}$ such that
$\,\langle A\hn^{m-1}\nh v,v\rangle=\ve\,$ for some $\,v\in\mv\nnh$. More
precisely, $\,\ve\,$ is the {\it sem\-i\-def\-i\-nite\-ness sign\/} of
$\langle A\hn^{m-1}\hn\cdot\hs,\hn\cdot\hh\rangle$, and
\begin{enumerate}
\item [(d)] vectors with $\,\langle A\hn^{m-1}\nh v,v\rangle=\ve\,$ form a
pair of opposite cosets of $\,\mv\hskip-3.5pt_1\w$ in $\,\mv\nnh$.
\end{enumerate}
We now prove, by induction on $\,j=1,\dots,m$, the existence of 
an ordered $\,j\hh$-tuple $\,(S\hn_1\w,\dots,S\nnh_j\w)
\in\mv\hskip-2.4pt/\mv\hskip-3.5pt_1\w\times\ldots
\times\mv\hskip-2.4pt/\mv\hskip-3.6pt_j\w$ of
cosets such that
$\,S\nnh_j\w\subseteq\ldots\subseteq S\hn_1\w$ while, 
for $\,\ve\,$ in (d) and every $\,v\in S\nnh_j\w$,
\begin{equation}\label{jtu}
\langle A\hn^{m-1}\nh v,v\rangle
=\ve\hh,\hskip5pt\langle A\hn^{m-2}\nh v,v\rangle
=\ldots=\langle A\hn^{m-j}\nh v,v\rangle=0\hh,
\end{equation}
along with uniqueness of $\,(S\hn_1\w,\dots,S\nnh_j\w)\,$ up to its
replacement by $\,(-S\hn_1\w,\dots,-S\nnh_j\w)$. As (d) yields our
claim for $\,j=1$, suppose that it holds for some $\,j-1\ge1$. Since
$\,\mv\hskip-3.6pt_j\w\subseteq\mv\hskip-3.6pt_{j-1}\w\subseteq\ldots
\subseteq\mv\hskip-3.5pt_1\w$, cf.\ (a),
\begin{enumerate}
\item [(e)] the spaces 
$\,\mv\hskip-3.6pt_{j-1}\w,\dots,\mv\hskip-3.5pt_1\w$ project onto sub\-spaces
$\,Q\hn_1\w,\dots,Q\nnh_{j-1}\w$ of dimensions
$\,1,\dots,j-1\,$ in the $\,j$-di\-men\-sion\-al quotient
$\,Q\nnh_j\w=\mv\hskip-2.4pt/\mv\hskip-3.6pt_j\w$,
\end{enumerate}
and $\,Q\hn_1\w\subseteq\ldots\subseteq Q\nnh_{j-1}\w$, 
while the cosets $\,S\nnh_{j-1}\w,\dots,S\hn_1\w$ of 
$\,\mv\hskip-3.6pt_{j-1},\dots,\mv\hskip-3.5pt_1\w$ in $\,\mv\nnh$, assumed to
exist (and be unique up to an overall sign), project onto an ascending chain of
cosets of $\,Q\hn_1\w,\dots,Q\nnh_{j-1}\w$ in $\,Q\nnh_j\w$. Let us fix a
vector $\,v\in S\nnh_{j-1}\w$, denote 
by $\,\hat R\hn_1\w,\dots,\hat R\nh_{j-1}\w$ the latter cosets (of dimensions 
$\,1,\dots,j-1$), and by 
$\,(\hh\cdot\hs,\hn\cdot\hh)\,$ the symmetric bi\-lin\-ear form on
$\,Q\nnh_j\w$ induced by $\,\langle A\hn^{m-j}\hn\cdot\hs,\hn\cdot\hh\rangle\,$ via
(b). Since (\ref{jtu}) is assumed to hold for our $\,v$, with $\,j\,$ replaced
by $\,j-1$, if we set $\,v\hn_i\w=A\hn^{j-i}\nh v$, $\,i=1,\dots,j$, then, 
for all $\,i,k\in\{1,\dots,j\}$, due to (c) and the first equality in this
version of (\ref{jtu}), $\,(v\hn_i\w,v\hn_k\w)=0$ if $\,i+k\le j$ and 
$\,(v\hn_i\w,v\hn_k\w)=\ve\,$ when $\,i+k=j+1$. The $\hs j\times j\hs$
matrix of these $\,(\hh\cdot\hs,\hn\cdot\hh)$-in\-ner products thus has the
entries all equal to $\,\ve\,$ on the main anti\-di\-ag\-o\-nal, and all zero
above it. Due to the resulting nondegeneracy of the matrix and the presence of
the zero entries, $\,v\hn_1\w,\dots,v\nh_j\w$ project onto a basis 
$\,\hat v\hn_1\w,\dots,\hat v\nh_j\w$ of $\,Q\nnh_j\w$, with
$\,\hat v\hn_i\w\in Q\nh_i\w$, $\,i=1,\dots,j$, and 
$\,(\hh\cdot\hs,\hn\cdot\hh)\,$ is a sem\-i-neu\-tral 
pseu\-\hbox{do\hs-}Euclid\-e\-an inner product in $\,Q\nnh_j\w$. 
Thus, $\,\hat v\hn_1\w\in Q\hn_1\w$ is 
$\,(\hh\cdot\hs,\hn\cdot\hh)$-or\-thog\-o\-nal to the basis
$\,\hat v\hn_1\w,\dots,\hat v\nh_{j-1}\w$ of $\,Q\nnh_{j-1}\w$, which makes
$\,Q\nnh_{j-1}\w$ the $\,(\hh\cdot\hs,\hn\cdot\hh)$-or\-thog\-o\-nal
complement of the $\,(\hh\cdot\hs,\hn\cdot\hh)$-null line $\,Q\hn_1\w$.
At the same time, the coset $\,\hat R\hn_1\w$ of $\,Q\hn_1\w$ is not contained
in the $\,(\hh\cdot\hs,\hn\cdot\hh)$-or\-thog\-o\-nal complement
$\,Q\nnh_{j-1}\w$ of $\,Q\hn_1\w$, since 
$\,(v\hn_1\w,v)=(A\hn^{j-1}\nh v,v)=\langle A\hn^{m-1}\nh v,v\rangle
=\ve\ne0\,$ in the $\,j-1\,$ version
of (\ref{jtu}), and so the vector $\,v=v\nh_j\w\in S\nnh_{j-1}\w$,
projecting onto $\,\hat v\nh_j\w\in\hat R\hn_1\w$, is not
$\,(\hh\cdot\hs,\hn\cdot\hh)$-or\-thog\-o\-nal to
$\,\hat v\hn_1\w$ spanning the line $\,Q\hn_1\w$. By Lemma~\ref{point},
$\,\hat R\hn_1\w$ intersects the $\,(\hh\cdot\hs,\hn\cdot\hh)$-null cone at
exactly one point, and so does $\,-\hat R\hn_1\w$. This ``point'' in the
$\,j$-di\-men\-sion\-al quotient
$\,Q\nnh_j\w=\mv\hskip-2.4pt/\mv\hskip-3.6pt_j\w$ is actually a coset
$\,S\nnh_j\w$ of $\,\mv\hskip-3.6pt_j\w$ in $\,\mv$, contained in
$\,S\nnh_{j-1}\w$, and its lying in the $\,(\hh\cdot\hs,\hn\cdot\hh)$-null
cone amounts to $\,\langle A\hn^{m-j}\nh v,v\rangle=0\,$ for all
$\,v\in S\nnh_j\w$, which establishes the inductive step and thus proves 
the existence and uniqueness claim about (\ref{jtu}).

This last claim, for $\,j=m$, yields a unique (up to a sign) coset
$\,S\nh_m\w$ of $\,\mv\hskip-2.4pt_m=\{0\}$, that is, a unique pair
$\,\{v,-v\}\,$ of opposite vectors in $\,\mv\nnh$, with
\begin{equation}\label{amo}
\langle A\hn^{m-1}\nh v,v\rangle
=\ve\hskip7pt\mathrm{and}\hskip7pt\langle A\hn^{i}\nh v,v\rangle
=0\hskip7pt\mathrm{whenever}\hskip7pti\ge0\,\mathrm{\ and\ }\,i\ne m-1\hh,
\end{equation}
the case of $\,i<m-1\,$ being due to (\ref{jtu}) for $\,j=m$, that of
$\,i\ge m\,$ immediate from (c). Note that $\,S\nh_m\w$ uniquely determines
the other cosets $\,S\nnh_j\w$ as
$\,S\nh_m\w\subseteq\ldots\subseteq S\hn_1\w$. Setting
$\,e\nh_i\w=A\hn^{m-i}\nh v$, $\,i=1,\dots,m$, we obtain an $\,m$-tuple of
vectors leading to matrices for $\,A\,$ and $\,\lr\,$ described in the
statement of the theorem, cf.\ (c) and (\ref{amo}). Nondegeneracy of the
latter matrix,
along with the abundance of zero entries in it, establishes both linear
independence of $\,e\nh_1\w,\dots,e\nh_m\w$ and the sem\-i-neu\-tral signature
of $\,\lr$. Uniqueness of $\,\{v,-v\}\,$ clearly implies uniqueness of
$\,e\nh_1\w,\dots,e\nh_m\w$ up to their replacement by 
$\,-e\nh_1\w,\dots,-e\nh_m\w$.

For the converse statement it suffices to note that the basis
$\,e\nh_1\w,\dots,e\nh_m\w$ has the
form $\,A\hn^{m-1}\nh v,A\hn^{m-2}\nh v,\dots,Av,v$, and so
self-ad\-joint\-ness of $\,A\,$ amounts to requiring that the matrix of
$\,\lr\,$ have a single value of the entries in each anti\-di\-ag\-o\-nal.
\end{proof}
\begin{corollary}\label{cmmte}The only linear iso\-metries of a 
pseu\-\hbox{do\hs-}Euclid\-e\-an space of dimension\/ $\,m\,$ commuting with
a given generic nil\-po\-tent self-ad\-joint en\-do\-mor\-phism\/ $\,A\,$ such 
that\/ $\,A\nh^{m-1}\nnh\ne0\,$ are\/ $\,\mathrm{Id}\,$ and\/ $\,-\mathrm{Id}$.
\end{corollary}
In fact, due the up-to-a-sign uniqueness of the basis in Theorem~\ref{gnnlp},
such a linear iso\-metry must transform this basis into itself or its
opposite.
\begin{corollary}\label{slfsm}Let\/ a nil\-po\-tent 
self-ad\-joint en\-do\-mor\-phism\/ $\,A\,$ of a pseu\-\hbox{do\hs-}Euclid\-e\-an
space\/ $\,\mv\hs$ have\/ $\,A\nh^{m-1}\nnh\ne0$, where\/ $\,m=\dim\mv\nnh$.
Then, for every\/ $\,q\in(0,\infty)$, there exists a unique pair\/
$\,\{\bc,-\bc\}\,$ of mutually opposite linear iso\-metries of\/ $\,\mv\hs$
with\/ $\,\bc\hskip-1.8ptA\bc^{-\nnh1}\nh=\hs q^2\hskip-2.3ptA$.

Such\/ $\,\bc\,$ is di\-ag\-o\-nal\-iz\-ed by a basis\/
$\,e\nh_1\w,\dots,e\nh_m\w$ chosen as in Theorem\/~{\rm\ref{gnnlp}}, with
the\/ respective eigen\-values\/ $\,q^{m-1}\nh,q^{m-3}\nh,\dots,q^{1-m}\nh$,
or their opposites, so that\/ 
$\,\bc\nh e\nnh_j\w=\pm q^{m+1-2\hn j}e\nnh_j\w\hskip2pt$ for\/
$\,j=1,\dots,m\,$ and some fixed sign\/ $\,\pm\hs$.
\end{corollary}
\begin{proof}
Uniqueness is immediate from Corollary~\ref{cmmte} since two such linear 
iso\-metries differ, com\-po\-si\-tion-wise, by one commuting with $\,A$.
Existence: defining the linear auto\-morphism $\,\bc\,$ by
$\,\bc\nh e\nh_j\w=\widetilde e\nh_j\w$, for
$\,\widetilde e\nh_j\w=q^{m+1-2\hn j}e\nh_j\w$,
we get the inner products
$\,\langle\widetilde e\nh_i\w,\widetilde e\nh_k\w\rangle
=\ve\hs\delta\nh_{i\hn j}\w$, and 
$\,q^2\hskip-2.3ptA\widetilde e\nh_j\w=\widetilde e\nh_{j-1}\w$,
for all $\,i,j\in\{1,\dots,m\}$, with $\,k=m+1-j\,$ and
$\,\widetilde e\nh_0\w=0$,
as required.
\end{proof}
\begin{remark}\label{nlpgn}For a nil\-po\-tent self-ad\-joint
en\-do\-mor\-phism $\,A\,$ of an $\,m$-di\-men\-sion\-al
pseu\-\hbox{do\hs-}Euclid\-e\-an space $\,\mv\nnh$, five
conditions are mutually equivalent:
\begin{enumerate}
\item[(i)] $A\nh^{m-1}\nnh\ne0$.
\item[(ii)] $\mathrm{rank}\hskip1.7ptA=m-1\,$ (in other words,
$\,\dim\,\mathrm{Ker}\hskip1.7ptA=1$).
\item[(iii)] $\pm\mathrm{Id}\,$ are the only linear self-iso\-metries of
$\,\mv\nnh$ commuting with $\,A$.
\item[(iv)] $A\,$ is generic (commutes with only finitely many linear
iso\-metries).
\item[(v)] $0\,$ is the only skew-ad\-joint en\-do\-mor\-phism of $\,\mv\nnh$
commuting with $\,A$.
\end{enumerate}
In fact, (i) yields (ii) due to (\ref{dmk}) -- (\ref{deo}), and the converse
is immediate as (ii) and (\ref{dog}) -- (\ref{dmk}) force all $\,d\hn_j\w$ to
equal $\,1$. The implications (i) $\implies$ (iii) $\implies$ (iv) $\implies$ 
(v) are obvious from Corollary~\ref{cmmte}. Finally, (v) implies (ii) as any
two vectors $\,v,v'\in\mathrm{Ker}\hskip1.7ptA\,$ are 
linearly dependent: the skew-ad\-joint en\-do\-mor\-phism $\,v\wedge v'\nh
=\langle v,\,\cdot\,\rangle\hs v'\nh-\langle v'\nh,\,\cdot\,\rangle\hs v$,
where $\,\lr\,$ is the inner product, commutes with $\,A$.
\end{remark}

\section{Invariant sub\-spaces}\label{is}
\setcounter{equation}{0}
This section uses the following assumptions and notations.

First, we fix $\,q\in(0,\infty)\smallsetminus\{1\}$, an integer $\,m\ge2$,
a generic self-ad\-joint nil\-po\-tent en\-do\-mor\-phism $\,A\,$ of an
$\,m$-di\-men\-sion\-al pseu\-\hbox{do\hs-}Euclid\-e\-an space $\,\mv\hs$ with
the inner product $\,\lr$, and a linear $\,\lr$-iso\-metry $\,\bc\,$ of
$\,\mv\hs$ having positive eigen\-values and satisfying the condition 
$\,\bc\hskip-1.8ptA\bc^{-\nnh1}\nh=\hs q^2\hskip-2.3ptA$.

According to Remark~\ref{nlpgn}, Theorem~\ref{gnnlp} and Corollary~\ref{slfsm},
the algebraic type of the above quadruple $\,(\mv\nh,\lr,A,\bc)\,$ is uniquely 
determined by $\,m,q\,$ and a sign parameter $\,\ve=\pm1$. More precisely,
we may choose a basis $\,e\nh_1\w,\dots,e\nh_m\w$ of $\,\mv\hs$ such that,
for some $\,\ve\in\{1,-\nnh1\}\,$ and all $\,i,j\in\{1,\dots,m\}$, 
with $\,e\nh_0\w=0\,$ and $\,k=m+1-j$,
\begin{equation}\label{bss}
Ae\nnh_j\w=e\nh_{j-1}\w\hh,\hskip20pt
\langle e\nh_i\w,e\nh_k\w\rangle=\ve\hs\delta\nh_{i\hn j}\w\hh,\hskip20pt
\bc\nh e\nnh_j\w=q^{m+1-2\hn j}e\nnh_j\w\hh.
\end{equation}
Let the operator $\,T\,$ act on functions $\,(0,\infty)\ni t\mapsto u(t)$, 
valued anywhere, by
\begin{equation}\label{tro}
[T\nnh u](t)\,=\,u(t/q)\hh.
\end{equation}
We also fix a $\,C^\infty\nnh$ function
\begin{equation}\label{qtf}
f:(0,\infty)\to\bbR\,\mathrm{\ \ with\ \ }\,q^2f(qt)\,
=\,f(t)\,\mathrm{\ whenever\ }\,t\in(0,\infty)\hh,
\end{equation}
and define $\,\wv\nh,\xe\hs$ to be the vector spaces of dimensions
$\,2\,$ and $\,2m\,$ formed by all real-val\-ued (or, $\,\mv\nh$-val\-ued)
functions $\,\y\,$ (or, $\,u$) on $\,(0,\infty)\,$ such that
\begin{equation}\label{ode}
\mathrm{i)}\hskip6pt\ddot\y\,=\,f\y\,\mathrm{\ \ or,\ respectively,\ \
}\mathrm{ii)}\hskip6pt\ddot u=fu+q^2\hskip-2.3ptAu\hh,
\end{equation}
with $\,(\hskip2.3pt)\hskip-2.2pt\dot{\phantom o}\nh=\hs d/dt$. The operator
$\,T\,$ obviously preserves $\,\wv\nh$, and so we may
select a basis $\,\y\nh^+\nnh,\y^-$ of the space of com\-plex-val\-ued 
solutions to (\ref{ode}-i) having
\begin{equation}\label{tye}
T\nnh\y\nh^+=\mu\nh^+\nnh\y\nh^+\mathrm{\ and\ }\,T\nnh\y^-\mathrm{\ equal\
to\ }\,\mu\hn^-\nnh\y^-\mathrm{\ plus\ a\ multiple\ of\ }\,\y\nh^+\nh,
\end{equation}
for some eigen\-values $\,\mu^\pm\nh\in\bbC$, the multiple being zero 
unless $\,\mu\nh^+\nnh=\mu^-\nnh\in\bbR$. Since the formula
$\,\alpha(\y\nh^+\nnh,\y^-)=\dot \y\nh^+ \y^-\nnh-\y\nh^+\dot \y^-\hs$ 
(a constant!) defines an area form on $\,\wv\hs$ such that 
$\,T^*\nh\alpha=q\nh^{-\nnh1}\nnh\alpha$, we have
$\,\det\hs T=q\nh^{-\nnh1}$ in $\,\wv\nh$. Consequently,
\begin{equation}\label{mpq}
\mu\nh^+\nnh\mu^-=\,\,q\nh^{-\nnh1}.
\end{equation}
In general, $\,\xe\hs$ is not preserved by either $\,T\,$ 
or by $\,\bc\,$ applied val\-ue\-wise via $\,u\mapsto\bc u$. Their composition
$\,\bc\hh T=T\hn\bc\,$ however, does leave $\,\xe$ invariant:
\begin{equation}\label{cte}
\bc\hh T:\xe\nh\to\xe,
\end{equation}
as it coincides with the operator $\,u\mapsto\sigma\hn u\,$ in (\ref{acs}).
The solution space $\,\xe\hs$ of (\ref{ode}-ii) has an 
ascending $\,m$-tuple of $\,\bc\hh T\nh$-in\-var\-i\-ant vector sub\-spaces:
\begin{equation}\label{asc}
\xe\nh_1\w\,\subseteq\,\ldots\,\subseteq\,\xe\nh_m\w\hs=\,\xe\,\mathrm{\
with\ }\,\dim\hh\xe\nnh_j\w=\hs2\hn j\hh,
\end{equation}
each $\,\xe\nnh_j\w$ consisting of solutions taking
values in the space $\,\mathrm{Ker}\hskip1.7ptA\hn^j\nh$. (Note that, 
as a consequence of (\ref{dmk}) -- (\ref{deo}),
$\,\dim\,\mathrm{Ker}\hskip1.7ptA\hn^j\nh=j$.)
\begin{theorem}\label{spctr}Given\/
$\,q,m,\mv\nnh,\lr,A,\bc\nh,e\nh_1\w,\dots,e\nh_m\w,T\nnh,f\nnh,\wv\nh,
\xe\nh,\y\nh^\pm\nnh,\mu^\pm\nh$ introduced earlier in this section, let\/
$\,\lz\,$ be any\/ $\,\bc\hh T\nh$-in\-var\-i\-ant sub\-space of\/ $\,\xe\nh$.

Then in some basis\/  
$\,u\nh^+_1\nnh,u^-_1,\dots,u\nh^+_m,u^-_m$ of the complexification\/ 
$\,\xe^\bbC$ of\/ $\,\xe\nh$, containing a basis of $\,\lz^\bbC\nnh$, the 
matrix of\/ $\,\bc\hh T\,$ is upper 
triangular with the diagonal 
\hbox{$(\gy\nh^+_1,\gy^-_1,\dots,\gy\nh^+_m,\gy^-_m)\,$} where, 
for some combination coefficients\/ $\,(0,\infty)\to\bbC$,
\begin{equation}\label{upm}
\gy^\pm_j\nnh=\nh q^{m+1-2\hn j}\mu^\pm\nnh\mathrm{\ and\
}\hs u\nh^\pm_j\nnh\mathrm{\ equals\ }\y\nh^\pm e\nnh_j\w\nh\mathrm{\ 
plus\ a\ combination\ of\ }\,e\nh_1\w,\dots,e\nnh_{j-1}\w,
\end{equation}
$j=1,\dots,m$. If\/ $\,\mu\nh^+\nnh,\mu^-\nnh\in\bbR$, we may replace
\hbox{`\hn com\-plex-val\-ued'} by `real-val\-ued' and the complexifications\/
$\,\bbC,\xe^\bbC,\xe^\bbC_{\nnh j}$ by the original real forms\/
$\,\bbR,\xe\nh,\xe\nnh_j\w$.
\end{theorem}
\begin{proof}The equation $\,\ddot u=f\nh u+Au\,$ imposed on 
$\,u=\y\nh_1\w e\nh_1\w+\ldots+\y\nnh_j\w e\nnh_j\w$, with $\,1\le j\le m\,$
and com\-plex-val\-ued functions $\,\y\nh_1\w,\dots,\y\nnh_j\w$, reads
\begin{equation}\label{eqn}
\ddot\y\nnh_j\w\hs=\,f\y\nnh_j\w\mathrm{\ \ and\ \
}\,\ddot\y\hn_i\w=f\y\hn_i\w+\y\hn_{i+1}\w\mathrm{\ for\ }\,i<j\hh.
\end{equation}
Since, by (\ref{bss}), $\,e\nh_1\w,\dots,e\nnh_j\w$ span 
$\,\mathrm{Ker}\hskip1.7ptA\hn^j\nh$, such $\,u\,$ lie in
$\,\xe^\bbC_{\nnh j}$, for $\,\xe\nnh_j\w$ appearing in (\ref{asc}), and we
can now define $\,u\nh^\pm_j$ by (\ref{upm}), declaring $\,\y\nnh_j\w$ in
(\ref{eqn}) to be $\,\y\nh^\pm$ and then solving the equations
$\,\ddot\y\hn_i\w=f\y\hn_i\w+\y\hn_{i+1}\w$ in the descending order
$\,i=j-1,\dots,1$, with a $\,2(j-1)$-di\-men\-sion\-al freedom of choosing
the functions $\,\y\hn_i\w$. As $\,u^\pm_j\notin\xe^\bbC_{\nh i}$ for 
$\,i<j$, the $\,2m\,$ solutions $\,u^\pm_j$ are linearly independent, and 
hence constitute a basis 
$\,u\nh^+_1\nnh,u^-_1,\dots,u\nh^+_m,u^-_m$ of $\,\xe^\bbC$ which makes
$\,\bc\hh T\,$ upper triangular with the required diagonal. More precisely, by
(\ref{bss}) -- (\ref{tye}), $\,\bc\hh T\nnh u\nh^+_j$ (or,
$\,\bc\hh T\nnh u^-_j$) equals $\,q^{m+1-2\hn j}\mu\nh^+\nh u\nh^+_j$ (or, 
$\,q^{m+1-2\hn j}\mu^-\nh u^-_j$ plus a multiple of 
$\,u\nh^+_j$), plus a linear combination of 
$\,u^\pm_i$ with $\,i<j$, the multiple being $\,0\,$ unless
$\,\mu\nh^+\nnh=\mu^-\nnh\in\bbR$. 

The freedom of choosing $\,\y\hn_i\w$ will now ensure that some
$\,u\nh^+_1\nnh,u^-_1,\dots,u\nh^+_m,u^-_m$ as above also contains a basis of
$\,\lz^\bbC\nnh$. Namely, for $\,\lz\nnh_j\w=\lz\cap\xe\nnh_j\w$, we get 
in\-clu\-sion-in\-duced, obviously injective operators
$\,\lz\nnh_j\w/\lz\nnh_{j-1}\w\to\xe\nnh_j\w/\xe\nnh_{j-1}\w$, where
$\,1\le j\le m\,$ and $\,\lz\nnh_0\w=\xe\nh_0\w=\{0\}$, so that, by
(\ref{asc}), $\,\delta\nh_j\w\in\{0,1,2\}$, with 
$\,\delta\nh_j\w=\dim\hs(\lz\nnh_j\w/\lz\nnh_{j-1}\w)$. Our $\,u\nh^\pm_j$ may
now be left completely arbitrary, as before, when $\,\delta\nh_j\w=0$. If
$\,j\,$ is fixed and $\,\delta\nh_j\w=2$, our operator 
$\,\lz\nnh_j\w/\lz\nnh_{j-1}\w\to\xe\nnh_j\w/\xe\nnh_{j-1}\w$ is an
isomorphism, and so the cosets of $\,u\nh^\pm_j$, forming a basis of 
$\,[\xe\nnh_j\w/\xe\nnh_{j-1}\w]^\bbC\nh$, are also realized as
$\,\lz^\bbC_{\nnh j-1}$ cosets of solutions in $\,\lz^\bbC_{\nnh j}$,
which we select as the required modified versions
of $\,u\nh^\pm_j$. Finally, in the case $\,\delta\nh_j\w=1$, the embedded line
$\,[\lz\nnh_j\w/\lz\nnh_{j-1}\w]^\bbC$ in
$\,[\xe\nnh_j\w/\xe\nnh_{j-1}\w]^\bbC\nh$, due to its
$\,\bc\hh T\nh$-in\-var\-i\-ance, must be one of the two eigen\-vector cosets 
represented by $\,u\nh^\pm_j$, and the latter can thus be modified (within our
$\,2(j-1)$-di\-men\-sion\-al freedom) so as to lie in
$\,\lz^\bbC_{\nnh j}$. Since
$\,\delta\nh_j\w=\dim\hs(\lz\nnh_j\w/\lz\nnh_{j-1}\w)$, the total number of
modified solutions, 
$\,\delta\hn_1\w+\ldots+\delta\hn_m\w$, equals $\,\dim\hs\lz$. Therefore,
they form a basis of $\,\lz^\bbC\nnh$.
\end{proof}

\section{$\mathrm{GL}\hh(\bbZ)$-pol\-y\-no\-mi\-als}\label{gp}
\setcounter{equation}{0}
By a {\it root of unity}, or a 
$\,\mathrm{GL}\hh(\bbZ)\hn${\it-pol\-y\-no\-mi\-al\/} we mean here any complex
number $\,z\,$ such that $\,z^k\nh=1\,$ for some integer $\,k\ge1\,$ or,
respectively, any polynomial of degree $\,d\ge1\,$ with integer coefficients,
the leading coefficient $\,(-\nnh1)^d\nh$, and the constant term $\,1\,$ or
$\,-\nnh1$. It is well known, cf.\ \cite[p.\ 75]{derdzinski-roter-10}, that
\begin{equation}\label{glz}
\begin{array}{l}
\mathrm{GL}\hh(\bbZ)\hyp\mathrm{pol\-y\-no\-mi\-als\hs\ of\hs\ degree\hs\ 
}\hs\,d\,\hs\mathrm{\hs\ are\hs\ precisely\hs\ the}\\
\mathrm{characteristic\ polynomials\ of\ matrices\ in\
}\hs\mathrm{GL}\hh(d,\bbZ).
\end{array}
\end{equation}
Every complex root $\,a\,$ of a
$\,\mathrm{GL}\hh(\bbZ)$-pol\-y\-no\-mi\-al $\,P\hs$ is an invertible
algebraic integer and $\,P\nh$, if also assumed irreducible, is the minimal
monic polynomial of $\,a$. Then, due to minimality, $\,a\,$ is not a root of
the derivative of $\,P\nh$, showing that
\begin{equation}\label{spr}
\mathrm{the\ complex\ roots\ of\ an\ irreducible\ 
}\,\mathrm{GL}\hh(\bbZ)\hyp\mathrm{pol\-y\-no\-mi\-al\
are\ all\ distinct.}
\end{equation}
{\it Irreducibility\/} is always meant here to be over $\,\bbZ\,$ or,
equivalently, over $\,\bbQ$.

We say that a
$\,\mathrm{GL}\hh(\bbZ)$-pol\-y\-no\-mi\-al has a {\it cyclic root group\/}
if its (obviously nonzero) complex roots generate a cyclic multiplicative
group of nonzero complex numbers. The goal of this section is
to show that
\begin{equation}\label{drg}
\begin{array}{l}
\mathrm{the\nh\ only\hn\ irreducible\hn\ 
}\mathrm{GL}\hh(\bbZ)\hyp\mathrm{pol\-y\-no\-mi\-als\hn\ with\hn\ a\hn\
cyclic}\\
\mathrm{root\hskip2pt\ group\hskip2pt\ are\hskip2pt\ the\hskip2pt\ cyclo\-tom\-ic\hskip2pt\ and\hskip2pt\ quadratic\hskip2pt\ ones.}
\end{array}
\end{equation}
We call an irreducible $\,\mathrm{GL}\hh(\bbZ)$-pol\-y\-no\-mi\-al {\it
cyclo\-tom\-ic\/} if all of its roots are roots of unity which, up to a sign, 
agrees with the standard terminology \cite{maier}. 
The cyclic root-group condition clearly does hold for all cyclo\-tom\-ic 
polynomials and all quadratic $\,\mathrm{GL}\hh(\bbZ)$-pol\-y\-no\-mi\-als. 

First, if an irreducible $\,\mathrm{GL}\hh(\bbZ)$-pol\-y\-no\-mi\-al $\,P\hs$ 
has among its roots $\,a\,$ and $\,a^k\nh$, for some
$\,a\in\bbC\smallsetminus\{1,-\nnh1\}\,$ and an integer
$\,k\notin\{0,1,-\nnh1\}$, then
\begin{equation}\label{rou}
\mathrm{every\ complex\ root\ of\ }\,P\hs\mathrm{\ is\ a\ root\ of\ unity.}
\end{equation}
In fact, if $\,k\ge2$, then, for such $\,P\nh,a$, all $\,\lambda\in\bbC$, all 
integers $\,r\ge1$, and some $\,\mathrm{GL}\hh(\bbZ)$-pol\-y\-no\-mi\-al $\,Q$,
\begin{equation}\label{plk}
P(\lambda\nh^{k^r})\,=\,Q(\lambda)\hs Q(\lambda\nh^k)\hh\ldots\hs
Q(\lambda\nh^{k\hh^{r-1}})\hh P(\lambda)\hh,
\end{equation}
as one sees using induction on $\,r$, the case $\,r=1\,$ being obvious as
$\,\lambda\mapsto P(\lambda\nh^k)$ has $\,a\,$ as a root, which makes it
divisible by the minimal polynomial $\,P\hs$ of $\,a$, and the induction step 
amounts to replacing $\,\lambda\,$ in (\ref{plk}) by $\,\lambda\nh^k\nh$.
Now (\ref{rou}) follows, or else $\,P$ would have infinitely many roots. 
The extension of (\ref{rou}) to negative integers $\,k\,$ is in turn immediate
if one notes that $\,(PQ)\hn^*\nh=P\hh^*\nnh Q^*$ and $\,P^{**}\nnh=P\hs$ for 
the {\it inversion\/} $\,P\hh^*$ of a degree $\,d\,$ polynomial 
$\,P\nh$, defined by $\,P\hh^*(\lambda)=\lambda\nh^d\nnh P(1/\lambda)\,$ or, 
equivalently, 
$\,P\hh^*\nh(\lambda)=a_0\w\lambda\nh^d\nh+\ldots+a_{d-1}\w\lambda+a_d\w$
whenever $\,P(\lambda)=a_0\w+a_1\w\lambda+\ldots+a_d\w\lambda\nh^d\nh$. More
precisely, we then replace (\ref{plk}) with
$\,P(\lambda\nh^{k^r})\,=\,Q^*\nh(\lambda)\hs Q(\lambda\nh^k)\hh\ldots\hs
Q^{[r]}(\lambda\nh^{k\hh^{r-1}})\hh P\hh^{[r]}(\lambda)$, where
$\,P\hh^{[r]}$ equals $\,P\hs$ or $\,P\hh^*$ depending on whether
$\,r\,$ is even or odd.
\begin{remark}\label{pwers}If a $\,\mathrm{GL}\hh(\bbZ)$-pol\-y\-no\-mi\-al
has the complex roots\/ $\,c_1\w,\dots,c_d\w$, and\/ $\,k$ is an integer, 
then\/ $\,c_1^{\hs k},\dots,c_d^{\hs k}$ are the roots of a\/  
$\,\mathrm{GL}\hh(\bbZ)$-pol\-y\-no\-mi\-al. (By (\ref{glz}), we may choose
the latter polynomial to be the characteristic polynomial of the $\,k\hh$th
power of a matrix in $\,\mathrm{GL}\hh(d,\bbZ)\,$ with the
characteristic roots $\,c_1\w,\dots,c_d\w$.)
\end{remark}
\begin{lemma}\label{roots}Let an irreducible\/ 
$\,\mathrm{GL}\hh(\bbZ)\hn$-pol\-y\-no\-mi\-al\/ $\,P\hs$ of degree\/ $\,d\,$
have a root\/ $\,a^k$ 
for some\/ $\,a\in\bbC\smallsetminus\{1,-\nnh1\}\,$ and an integer\/ 
$\,k\ne0$. Then
\begin{equation}\label{inv}
\,a\,\mathrm{\ is\ an\ invertible\ algebraic\ integer}
\end{equation}
having some\/ 
$\,\mathrm{GL}\hh(\bbZ)\hn$-pol\-y\-no\-mi\-al\/ $\,S\,$ as its
minimal polynomial, and the complex roots\/ $\,c_1\w,\dots,c_r\w$ of\/ $\,S\,$
can be rearranged so that, with\/ $\,d\le r$,
\begin{equation}\label{rop}
P(\lambda)\,=\,(c_1^{\hs k}\nh-\lambda)\ldots(c_d^{\hs k}\nh-\lambda)\,\,
\mathrm{\ and\ }\,\,\{c_1^{\hs k},\dots,c_d^{\hs k}\}
=\{c_1^{\hs k},\dots,c_r^{\hs k}\}\hh,
\end{equation}
\end{lemma}
\begin{proof}If $\,k>0$, the polynomial $\,\lambda\mapsto P(\lambda\nh^k)\,$
has the root $\,a$, which yields (\ref{inv}) and the equality
$\,P(\lambda\nh^k)=Q(\lambda)\hs S(\lambda)\,$ for all $\,\lambda\in\bbC\,$
and some $\,\mathrm{GL}\hh(\bbZ)$-pol\-y\-no\-mi\-al\/ 
$\,Q$. Thus, the $\,k\hh$th powers of all the roots $\,c_1\w,\dots,c_r\w$ of 
$\,S\,$ are roots of $\,P\nh$. The polynomial $\,R\,$ with the roots 
$\,c_1^{\hs k},\ldots,c_r^{\hs k}$ is a
$\,\mathrm{GL}\hh(\bbZ)$-pol\-y\-no\-mi\-al (Remark~\ref{pwers}), while each
factor in its unique irreducible factorization has simple roots by
(\ref{spr}), which are also roots of $\,P\nh$, and irreducibility of $\,P\hs$
thus implies that the factor must equal $\,P\nh$. In other words, $\,R\,$ is 
a power of $\,P\nh$, and (\ref{rop}) follows. When $\,k<0$, the preceding
assumptions (and conclusions) hold with $\,k,P\hs$ replaced by 
$\,|k|,P\hh^*$ (and $\,a,S\,$ unchanged), so that 
$\,P\hh^*\nh(\lambda)
=(c_1^{\hs|k|}\nh-\lambda)\ldots(c_d^{\hs|k|}\nh-\lambda)$, as required in
(\ref{rop}).
\end{proof}
\begin{lemma}\label{powrs}If an irreducible\/ 
$\,\mathrm{GL}\hh(\bbZ)\hn$-pol\-y\-no\-mi\-al\/ $\,P\hs$ has two roots of the 
form\/ $\,a^k$ and\/ $\,a^\ell$ for\/
$\,a\in\bbC\smallsetminus\{1,0,-\nnh1\}\,$ and two distinct nonzero integers\/ 
$\,k,\ell\ge2$, then all roots of\/ $\,P\hs$ have modulus\/ $\,1$.
\end{lemma}
\begin{proof}Let $\,k>\ell$. The two versions of (\ref{rop}), one for $\,k\,$
and one for $\,\ell$, involve the same roots $\,c_1\w,\dots,c_r\w$ of the 
same polynomial $\,S$, so that
\begin{equation}\label{mod}
\{|\hh c_1\w|^{\hn k},\dots,|\hh c_r\w|^{\hn k}\}\,
=\,\{|\hh c_1\w|^{\hn\ell},\dots,|\hh c_r\w|^{\hn\ell}\}\hh.
\end{equation}
If the greatest (or, least) of the moduli $\,|\hh c_1\w|,\dots,|\hh c_r\w|\,$
were greater (or, less) than $\,1$, its $\,k\hh$th (or, $\,\ell\hh$th) power
would lie on the left-hand (or, right-hand) side of (\ref{mod}) and be greater
than any number on the opposite side, contrary to the equality in
(\ref{mod}). Thus, $\,|\hh c_1\w|=\ldots=|\hh c_r\w|=1$.
\end{proof}
\begin{lemma}\label{mdone}If all roots of an irreducible\/  
$\,\mathrm{GL}\hh(\bbZ)\hn$-pol\-y\-no\-mi\-al\/ $\,P\hs$ have modulus\/ 
$\,1$, then they are roots of unity, that is, $\,P\hs$ is cyclo\-tom\-ic.
\end{lemma}
\begin{proof}A matrix in $\,\mathrm{GL}\hh(d,\bbZ)\,$ with 
the characteristic polynomial $\,P\nh$, cf.\ (\ref{glz}), treated as
an auto\-mor\-phism of $\,\bbC^d$ is, in view of (\ref{spr}),
di\-ag\-o\-nal\-iz\-ed
by a suitable basis, with unit diagonal entries, so that its powers form a
bounded sequence, with a convergent sub\-se\-quence. As these powers preserve
the real form $\,\bbR\nh^d\nh\subseteq\bbC^d\nh$, the convergence takes place
in $\,\mathrm{GL}\hh(d,\bbR)\,$ and discreteness of the subset
$\,\mathrm{GL}\hh(d,\bbZ)$ makes the subsequence ultimately constant.
\end{proof}
\begin{proof}[Proof of (\ref{drg})]Consider an irreducible 
$\,\mathrm{GL}\hh(\bbZ)$-pol\-y\-no\-mi\-al with a cyclic root 
group generated by $\,a\in\bbC$. By (\ref{spr}), we may assume that 
$\,a\notin\{1,0,-\nnh1\}$. If $\,a\,$ is (or is not) a root, 
our claim follows from (\ref{rou}) (or, Lemmas~\ref{powrs} --~\ref{mdone}).
\end{proof}

\section{The combinatorial argument}\label{ca}
\setcounter{equation}{0}
The main result of this section, Theorem~\ref{noset}, will serve as the final
step needed to prove Theorem~\ref{modif} in Section~\ref{pt}.

Any $\,m,\r\in\bbZ\,$ with $\,m\ge2\,$ give rise to functions
$\,\ep,\vf:\bbZ\to\bbZ\,$ and integers $\,a_0\w,a_1\w$ such that, for any
$\,a,b\in\bbZ$,
\begin{equation}\label{aph}
\begin{array}{rl}
\mathrm{i)}&\hskip5pt\ep(a)=m-(-\nnh1)^a\hn\r-a\hh,\hskip15pt
\mathrm{ii)}\hskip5pt\vf(a)=2m-2(-\nnh1)^a\hn\r-a\hh,\\
\mathrm{iii)}&\hskip5pt\ep\,\mathrm{\ is\ bijective\ and\ }\,\vf\,\mathrm{\
is\ an\ involution,}\\
\mathrm{iv)}&\hskip5pt\ep^{-\nnh1}\nh(b)=m-(-\nnh1)^{m+\r+b}\hskip-.3pt\r-b\hh,
\hskip12pt\mathrm{v)}\hskip5pt\vf(a)=\ep^{-\nnh1}\nh(-\ep(a))\hh,\\
\mathrm{vi)}&\hskip5pta_1\w=\ep^{-\nnh1}\nh(1)=m+(-\nnh1)^{m+\r}\nh\r-1\hh,\\
\mathrm{vii)}&\hskip5pta_0\w=\ep^{-\nnh1}\nh(0)=m-(-\nnh1)^{m+\r}\nh\r\hh,
\hskip18pt\mathrm{viii)}\hskip5pta_0\w+a_1\w=2m-1\hh.
\end{array}
\end{equation}
Let integers $\,m\ge2\,$ and $\,\r\,$ be fixed, 
$\,\vr\nh=\{1,\dots,2m\}$, and $\,|\hskip2.3pt|\,$ denote cardinality.
\begin{theorem}\label{noset}There is no set\/ $\,\zs\subseteq\vr\hs$ with the
following properties.
\begin{enumerate}
\item[(a)] $a_1\w\in\zs\,$ and\/ $\,\vf(a_1\w)\notin\zs$.
\item[(b)] $a_0\w\in\zs\,$ if and only if\/ $\,m\,$ is even.
\item[(c)] If\/ $\,a,b\in\vr\hs$ and\/ $\,a+b=2m+1$, then exactly one of\/
$\,a,b\,$ lies in\/ $\,\zs$.
\item[(d)]For every\/ $\,a\in\zs\smallsetminus\{a_1\w\}\,$ there exists\/
$\,b\in\zs\,$ with $\,\ep(b)=-\ep(a)$.
\item[(e)] $|\hs\zs\cap\{1,2,\dots,2j\}|\le j\,$ whenever\/ 
$\,j\in\{1,\dots,m\}$.
\end{enumerate}
\end{theorem}
\begin{proof}Equivalently, (c) states that $\,\zs\,$ is a selector for the 
$\,m$-el\-e\-ment family $\,\{\{a,b\}\subseteq\vr:a+b=2m+1\}$. Hence
$\,|\zs|=m$. In addition,
\begin{equation}\label{bmt}
\mathrm{i)}\hskip6pt|\zs|\,=\,m\,\ge\,3\hh,\hskip11pt
\mathrm{ii)}\hskip6pt|\hh\r|\,\le\,m\,-\,1\hh,\hskip11pt
\mathrm{iii)}\hskip6pt\vf(\zs\smallsetminus\{a_1\w\})
=\zs\smallsetminus\{a_1\w\}\hh.
\end{equation}
In fact, as $\,a_1\w\ne a_0\w$ and $\,a_0\w+a_1\w=2m-1\,$ by
(\ref{aph}-viii), having $\,m=2$ in (\ref{bmt}-i) would,
by (a) -- (b), give $\,\zs=\{a_0\w,a_1\w\}\subseteq\{1,2,3,4\}\,$ and
$\,a_0\w+a_1\w=3$, implying that $\,\zs=\{1,2\}$, contrary to (e) for 
$\,j=1$. Next, (d) and (\ref{aph}-v) give 
$\,\vf(\zs\smallsetminus\{a_1\w\})\subseteq\zs\smallsetminus\{a_1\w\}$, cf.\
(\ref{aph}-iii), with the image not containing $\,a_1\w$, as otherwise, by 
(\ref{aph}-iii), $\,\vf(a_1\w)\,$ would lie in $\,\zs$, which contradicts
(\ref{aph}-i); and (\ref{aph}-iii) makes the inclusion an equality, proving
(\ref{bmt}-iii). Finally,
using (\ref{bmt}-i), we may fix $\,a\in\zs\smallsetminus\{a_1\w,2m\}$. Thus,
by (\ref{bmt}-iii) and (\ref{aph}-ii), $\,1\le\vf(a)=2m-2(-\nnh1)^a\r-a\le2m$. 
When $\,a\,$ is even (odd) this becomes $\,2\le2m-2\r-a\le2m\,$ (or, 
$\,1\le2m+2\r-a\le2m-1$), yielding $\,1-m\le\r\le m-2\,$ (or 
$\,1-m\le\r\le m-1$), and (\ref{bmt}-ii) follows.

Let us now define $\,c_\pm\w\in\bbZ\,$ by
\begin{equation}\label{cpm}
c_\pm\w=m\mp\r\mathrm{,\ so\ that\ }\,1\le\,c_\pm\w\le\,2m-1\,\mathrm{\ due\
to\ (\ref{bmt}}\hyp\mathrm{ii),}
\end{equation}
denote by $\,\vr\nnh_\pm\w$ (or, $\,\zs_\pm\w$) the set of all $\,a\in\vr\hs$
(or, $\,a\in\zs\smallsetminus\{a_1\w\}$) having $\,(-\nnh1)^a\nh=\pm1$ and,
finally, given $\,a,b\in\vr\nnh_\pm\w$ with $\,a\le b$, set
$\,[a,b\hh]_\pm\w=[a,b\hh]\cap\vr\nnh_\pm\w$, referring to any such
$\,[a,b\hh]_\pm\w$ as an {\it even$/$\hskip-1.4ptodd sub\-in\-ter\-val\/} of
$\,\vr\nh$. Finally, we let 
$\,\zr_\pm\w$ stand for the maximal even/odd sub\-in\-ter\-val of $\,\vr\hs$
which is symmetric about $\,c_\pm\w$. Then
\begin{equation}\label{beb}
\begin{array}{rl}
\mathrm{i)}&\hskip5pt\zs=\zs_+\w\cup\zs_-\w\cup\{a_1\w\}\hh,\hskip17pt
\vf(\zs_\pm\w)=\zs_\pm\w\hh,\hskip17pt\zs_\pm\w\subseteq\zr_\pm\w\hh,\\
\mathrm{ii)}&\hskip5pt\zr_+\w=[2,2m-2\r-2]_+\w\hh,\hskip9pt
\zr_-\w=[2\r+1,2m-1]_-\w\hskip9pt\mathrm{if\ }\,\r\ge0\hh,\\
\mathrm{iii)}&\hskip5pt\zr_+\w=[-\nh2\r,2m]_+\w\hh,\hskip9pt
\zr_-\w=[1,2m+2\r-1]_-\w\hskip9pt\mathrm{if\ }\,\r<0\hh,\\
\mathrm{iv)}&\hskip5pt\vf\,\mathrm{\ restricted\ to\ even/odd\ integers\ is\
the\ reflection\ about\ }\,c_\pm\w.
\end{array}
\end{equation}
In fact, the first relation in (\ref{beb}-i) is obvious, the second immediate
from (\ref{bmt}-iii) since, by (\ref{aph}-ii), $\,\vf:\bbZ\to\bbZ\,$ preserves
parity, Also, (\ref{aph}-ii) yields (\ref{beb}-iv), which in turns shows that 
$\,\zs_\pm\w=\vf(\zs_\pm\w)\,$ is a (possibly empty) union
of sets $\,\{a,b\}\,$ having $\,c_\pm\w$ as the midpoint, and so
$\,\zs_\pm\w\subseteq\zr_\pm\w$. Finally, depending on whether
$\,c_\pm\w=m\mp\r$ is less (or, greater) than the midpoint $\,m+1/2\,$ of
$\,\vr\nh$, one endpoint of $\,\zr_\pm\w$ must lie in $\,\{1,2\}\,$ (or, in
$\,\{2m-1,2m\}$), and the other endpoint added to this one must yield
$\,2c_\pm\w$, which proves (\ref{beb}-ii) -- (\ref{beb}-iii).

Note that, as an obvious consequence of (\ref{beb}),
\begin{equation}\label{fls}
\begin{array}{l}
\zs\smallsetminus\{a_1\w\}\,\mathrm{\ fails\ to\ include\ specific\ integers\
from\ }\,\vr\nh\mathrm{,\ which\ are\hskip-3pt:}\\
\mathrm{the\ lowest\ }\,\r\,\mathrm{\ odd\ and\ highest\
}\,\r+1\,\mathrm{\ even\ ones\ when\ }\,\r>0\hh,\\
\mathrm{the\ highest\ }\,|\hh\r|\,\mathrm{\ odd\ and\ lowest\ 
}\,|\hh\r|-1\,\mathrm{\ even\ ones\ for\ }\,\r<0\hh,\\
\mathrm{the\ integer\ }\,2m\, \mathrm{\ if\ }\,\r=0\hh.
\end{array}
\end{equation}
Furthermore, one necessarily has
\begin{equation}\label{rzo}
\r\,\in\,\{0,\hs-\nnh1\}\hh.
\end{equation}
To see this, we begin by excluding the possibility that $\,\r\ge2\,$ (or,
$\,\r\le-\hn3$). Namely, if this was the case, (\ref{fls}) would give
$\,1,3,2m-2,2m\notin\zs\smallsetminus\{a_1\w\}\,$ (when $\,\r\ge2$), or 
$\,2,4,2m-3,2m-1\notin\zs\smallsetminus\{a_1\w\}\,$ (for $\,\r\le-\hn3$). From
the two pairs $\,\{1,2m\},\,\{3,2m-2\}$ (or, $\,\{2,2m-1\},\,\{4,2m-3\}$)
we would choose one, $\,\{a,b\}$, having $\,a_1\w\notin\{a,b\}\,$ and
$\,a+b=2m+1$, as well as $\,a,b\notin\zs$, which contradicts (c).

The next two cases that need to be excluded are $\,\r=1\,$ and $\,\r=-\nh2$. If
one of them occurred, (\ref{fls}) would give
$\,1,2m\notin\zs\smallsetminus\{a_1\w\}\,$ (if $\,\r=1$), or 
$\,\,2,2m-1\notin\zs\smallsetminus\{a_1\w\}$ (for $\,\r=-\nh2$), which would
again contradict (c), unless $\,a_1\w\in\{1,2m\}\,$ and $\,\r=1$, or
$\,a_1\w\in\{2,2m-1\}\,$ and $\,\r=-\nh2$. However, each of the resulting
four possible values $\,(1,1),(2m,1),(2,-\nh2),(2m-1,-\nh2)\,$ for the
ordered pair $\,(a_1\w,\r)\,$ leads, via (\ref{aph}-vi), to the immediate
conclusion that $\,m\le1$, contrary to (\ref{bmt}-i), and so 
(\ref{rzo}) follows.

As the next step, we write $\,m=2\hn j\,$ ($m\,$ even) or $\,m=2\hn j+1\,$
($m\,$ odd), so that $\,j\ge1\,$ by (\ref{bmt}-i), and proceed to establish
the inclusion
\begin{equation}\label{icl}
\zs'\cup\{a_*\w\}\subseteq\zs\cap\{1,2,\dots,2j\}\mathrm{,\ with\
}\,|\hs\zs'\cup\{a_*\w\}|=j+1\hh,
\end{equation}
which will contradict (e), thus completing the proof of the theorem. Here
$\,\zs'$ is the $\,j$-el\-e\-ment set consisting of all integers from
$\,\{1,2,\dots,2j\}\,$ with a specific parity (even if $\,\r=-\nnh1$, odd for
$\,\r=0$), and $\,a_*\w=a_0\w$ ($m\,$ even) or $\,a_*\w=a_1\w$ ($m\,$ odd).

To derive (\ref{icl}), we list various conclusions in two separate columns 
(one for either possible value of $\,\r$):
\[
\begin{array}{rll}
\mathrm{(A)}\hskip10pt&\r=0&\r=-\nnh1\\
\mathrm{(B)}\hskip10pt&\zs'\nh=\{1,3,\dots,2\hn j-1\}\hskip20pt&\zs'\nh
=\{2,4,\dots,2\hn j\}\\
\mathrm{(C)}\hskip10pt&a_*\w=2\hn j\in\zs&a_*\w=2\hn j-1\in\zs\\
\mathrm{(D)}\hskip10pt&a_1\w=m-1&a_1\w=m-1+(-\nnh1)^m\\
\mathrm{(E)}\hskip10pt&a_0\w=m&a_0\w=m-(-\nnh1)^m\\
\mathrm{(F)}\hskip10pt&2m\notin\zs&2m-1\notin\zs\\
\mathrm{(G)}\hskip10pt&1\in\zs&2\in\zs\hh.
\end{array}
\]
In fact, (B) is the definition of $\,\zs'\nh$, (E), (D), (C) follow from
(\ref{aph}-vii) -- (\ref{aph}-viii), with $\,a_*\w\in\zs\,$ due to (a) -- (b),
while (F) is immediate from (\ref{fls}) for $\,\r\in\{-\nnh1,0\}$, and (G)
from (F) and (c). What still remains to be shown, for (\ref{icl}), is the
inclusion
\begin{equation}\label{bpi}
\zs'\hs\subseteq\,\zs\hh,
\end{equation}
as (\ref{bpi}) combined with (B) -- (C) obviously yields (\ref{icl}).

To this end, consider $\,\varPsi:\bbZ\to\bbZ\,$ given by
$\,\varPsi(a)=2m+1-a$, so that (c) amounts to
$\,|\hs\zs\cap\{a,\varPsi(a)\}|=1\,$ for all $\,a\in\vr\hs$ or, equivalently, 
$\,\varPsi(\zs)=\vr\smallsetminus\zs\,$ and 
$\,\varPsi(\vr\smallsetminus\zs)=\zs$. 
Now, in our case, given an integer $\,i$,
\begin{equation}\label{ind}
\mathrm{if\ }\,1\le\,i<m-2\,\mathrm{\ and\ }\,i\in\zs\mathrm{,\ then\
}\,i+2\in\zs\hh.
\end{equation}
Namely, for the sign $\,\pm\,$ such that $\,(-\nnh1)^i\nh=\pm1$,
(\ref{bmt}-iii) and (\ref{beb}-iv) yield
\[
\mathop{i}\limits_{\textstyle{\mathrm{in}}}
\mathop{\,\,-\!\!\!\longrightarrow\,\,}
\limits_{\textstyle{\vf}}
\mathop{2m\mp2\r-i}\limits_{\textstyle{\mathrm{in}}}
\mathop{\,\,-\!\!\!\longrightarrow\,\,}
\limits_{\displaystyle{\varPsi}}
\mathop{i\pm2\r+1}\limits_{\textstyle{\mathrm{out}}}
\mathop{\,\,-\!\!\!\longrightarrow\,\,}
\limits_{\textstyle{\vf}}
\mathop{2m-i-1}\limits_{\textstyle{\mathrm{out}}}
\mathop{\,\,-\!\!\!\longrightarrow\,\,}
\limits_{\displaystyle{\varPsi}}
\mathop{i+2}\limits_{\textstyle{\mathrm{in}}}\hn,
\]
`in' or `out' meaning lying in $\,\zs\,$ or in $\,\vr\smallsetminus\zs$. In
fact, the four sums of pairs of adjacent integers in the above displayed line
are $\,2(m\mp\r)=2c_\pm\w$, $\,2m+1$, $\,2(m\pm\r)=2c_\mp\w$, $\,2m+1$, as
required in the definitions of the reflections $\,\varPsi\,$ and $\,\vf$, the
latter restricted to even/odd integers. On the other hand, the inequality
$\,i<m-2\,$ implies, via (D), that $\,i\ne a_1\w\ne2m-i-1\,$ (and so
$\,2m-i-1\notin\zs$, for otherwise $\,i\pm2\r+1=\vf(2m-i-1)\,$ would lie in
$\,\zs$).

Now (\ref{ind}) combined with (G) and (B) proves (\ref{bpi}) by induction on
$\,i$. Specifically, the highest value of odd (or, even) $\,i\,$ such that 
this yields $\,i\in\zs\,$ is the one with $\,i-2<m-2\le i$, which is the
required value $\,2\hn j-1\,$ (or, $\,2\hn j$) except for even $\,m\,$ and
$\,\r=-\nnh1$. In the latter case, although we get $\,2\hn j-2\,$ instead of 
$\,2\hn j=m$, we have $\,2\hn j=m=a_1\w\in\zs\,$ nevertheless, due to (D) and
(a).
\end{proof}

\section{Proof of Theorem~\ref{modif}}\label{pt}
\setcounter{equation}{0}
We argue by contradiction. Suppose that, for some rank-one ECS model manifold 
$\,(\hm\nh,\hg)\,$ defined by (\ref{met}), with (\ref{dta}), and for $\,\gp\,$
as in Theorem~\ref{isogp}, there exists
\begin{equation}\label{sgp}
\begin{array}{l}
\mathrm{a\ sub\-group\ }\hs\Gm\nh\subseteq\nnh\gp\hs\mathrm{\ acting\ on\ 
}\hs\hm\nh\mathrm{\ freely\ and\ properly\ discontin}\hyp\\
\mathrm{u\-ous\-ly\ with\ a\ generic\ compact\ quotient\ manifold\
}\,M\hs=\,\hm\nnh/\hh\Gm\nnh,
\end{array}
\end{equation}
yet $\,K\hskip-1.5pt_+\w$ in (\ref{ppr}) is infinite cyclic. As
$\,K\hskip-1.5pt_+\w=K\nh\cap\hs(0,\infty)$, by Lemma~\ref{image}, for the
image $\,K\hs$ of the homo\-mor\-phism
$\,\Gm\ni(q,p,\bc\nh,r,u)\mapsto q$, we get (\ref{dch}-b).
Theorem~\ref{dilat} now allows us to set $\,I\nh=(0,\infty)\,$ in
(\ref{dta}), and all $\,(q,p,\bc\nh,r,u)\in\Gm\hs$ have $\,p=0$. We fix
\begin{equation}\label{gen}
\hga\,=\,(q,0,\bc,\hat r,\hat u)\in\Gm\,\mathrm{\ such\ that\
}\,q\,\mathrm{\ is\ a\ generator\ of\ }\,K\hskip-1.5pt_+\w. 
\end{equation}
From (\ref{frs}) and Theorem~\ref{isogp}, we have
(\ref{qtf}) and $\,\bc\hskip-1.8ptA\bc^{-\nnh1}\nh=\hs q^2\hskip-2.3ptA$, 
for $\,f\nh,A\,$ in (\ref{dta}). Using the notations of (\ref{tro}) --
(\ref{cte}), with $\,m=n-2$, we replace $\,\Gm\nh$, without loss of generality,
by a fi\-nite-in\-dex sub\-group $\,\Gm\hskip-2.3pt_+\w$, which 
allows us to assume that
\begin{equation}\label{wlo}
q\in(0,\infty)\smallsetminus\{1\}\hh,\hskip9pt\bc\,\mathrm{\ has\ positive\
eigen\-values,\ and\ }\,\mu^\pm\nh\in\bbC\smallsetminus(-\infty,0\hs]\hh.
\end{equation}
Namely, each of these additional requirements amounts to passing from
$\,\Gm\hs$ to a sub\-group of index at most $\,2\,$ (or, equivalently, from
$\,M\,$ to the corresponding finite isometric covering). Specifically, we
successively intersect
$\,\Gm\hs$ with the kernels of the  homo\-mor\-phisms
$\,\Gm\to\{1,-\nnh1\}\,$ sending $\,(q,0,\bc\nh,r,u)\,$ to
$\,\mathrm{sgn}\,q\,$ and $\,\mathrm{sgn}\,\bc$, the latter sign accounting 
for positivity of negativity of the eigen\-values of $\,\bc$. 
(According to Corollary~\ref{slfsm}, one of these cases must take place, and 
all $\,\bc\,$ occurring in $\,\gp\,$ form an Abel\-i\-an group.) The
last condition (positivity of $\,\mu^\pm$ when they are real) is achieved by
replacing $\,\hga,q,\bc,\mu^\pm$ with their squares and $\,\Gm\hs$ with the
corresponding homo\-mor\-phic pre\-im\-age of the in\-dex-two sub\-group of
$\,K\hskip-1.5pt_+\w$ generated by $\,q^2\nh$, which is to be done only if
$\,\mu^\pm$ are real and negative, cf.\ (\ref{mpq}). Finally, we define a
linear operator $\,\varPi:\bbR\times\xe\nh\to\bbR\times\xe\hs$ by
\begin{equation}\label{pru}
\varPi(r,u)\,
=\,(2\hh\varOmega(\bc\hh T\nnh u,\hat u)\,+\,r/q,\,\bc\hh T\nnh u)\hh.
\end{equation}
From the assumption that $\,K\hskip-1.5pt_+\w$ is infinite cyclic we will
derive, in Lemma~\ref{gnrtr}, the existence of a vector sub\-space
$\,\lz\subseteq\xe\hs$ having the following properties.
\begin{equation}\label{ace}
\begin{array}{rl}
\mathrm{(A)}&\dim\lz=m\mathrm{,\ where\ }\,m=n-2\hh.\\
\mathrm{(B)}&\bc\hh T\,\mathrm{\ leaves\ }\,\lz\,\mathrm{\ invariant.}\\
\mathrm{(C)}&\varPi(\Sigma')=\Sigma'\mathrm{\ for\ some\ lattice\
}\,\Sigma'\mathrm{\ in\ }\,\bbR\times\lz\hh.\\
\mathrm{(D)}&\varOmega(u,u')=0\,\mathrm{\ whenever\ }\,u,u'\in\lz\hh.\\
\mathrm{(E)}&u\mapsto u(t)\,\mathrm{\ is\ an\ iso\-mor\-phism\
}\,\lz\to\mv\hs\mathrm{\ for\ every\ }\,t\in(0,\infty)\hh.
\end{array}
\end{equation}
\begin{remark}\label{eqdif}For any rank-one ECS model manifold (\ref{dta}) --
(\ref{met}), with $\,\hp$ and the solution space $\,\xe\hs$ defined in
(\ref{ghm}) and (\ref{vss}), if a vector sub\-space $\,\lz\subseteq\xe$
satisfies (\ref{ace}-E), with any $\,I\hs$ instead of $\,I\nh=(0,\infty)$,
then, restricting (\ref{prj}) to $\,(0,\infty)\times\bbR\times\lz$ we
clearly obtain an $\,\hp$-equi\-var\-i\-ant dif\-feo\-mor\-phism 
\[
I\nh\times\bbR\times\lz\,\to\,\hm\hs
=\,I\nh\times\bbR\times\mv\nh,
\]
its bijectivity being due to (\ref{ace}-E), and smoothness of 
its inverse -- to the smooth dependence of the iso\-mor\-phism 
$\,\lz\ni u\mapsto u(t)\in\mv\hs$ on $\,t\,$ along with
real-an\-a\-lyt\-ic\-i\-ty of the iso\-mor\-phism-in\-ver\-sion operation.
\end{remark}
\begin{lemma}\label{gnrtr}A vector sub\-space\/ $\,\lz\subseteq\xe\hs$ with\/
{\rm(\ref{ace})} exists if the conditions preceding\/ {\rm(\ref{pru})} are all 
satisfied.
\end{lemma}
\begin{proof}The surjective sub\-mer\-sion 
$\,\hm\ni(t,s,v)\mapsto(\log t)/(\log q)\in\bbR$, being clearly 
equi\-var\-i\-ant relative to the homo\-mor\-phism
\begin{equation}\label{hmz}
\Gm\hskip-2.3pt_+\w\ni\gamma'\nh=(q'\nh,0,\bc'\nh,r'\nh,u')
\mapsto(\log q')/(\log q)\in\bbZ
\end{equation}
along with the 
obvious actions of $\,\Gm$ on $\,\hm\nh$, via (\ref{act}) with $\,p=0$, and
$\,\bbZ\,$ on $\hs\bbR\hs$ by translations, descends to a surjective
sub\-mer\-sion $\,M\to S^1$ which is
\begin{equation}\label{dsc}
\mathrm{a\ bundle\ projection\ }\,\hm\nnh/\hh\Gm\hskip-2.3pt_+\w\to\bbR/\hs\bbZ
=S^1\nh.
\end{equation}
according to Remark~\ref{bndpr}. The kernel $\,\Sigma\,$ of (\ref{hmz})
equals 
$\,\Sigma=\{(1,0,\mathrm{Id})\}\times\Sigma'$ for some set 
$\,\Sigma'\nh\subseteq\bbR\times\xe\nh$, since $\,\bc'\nh$ in (\ref{hmz}),
due to its positivity, (\ref{frs}) and Corollary~\ref{slfsm}, is uniquely
determined by
$\,q'\nh$. Thus, $\,\Sigma\subseteq\hp$, for $\,\hp\,$ given by (\ref{ghm}).
As a consequence of Lemma~\ref{cplfl}(b) and (f) in Section~\ref{ro}, the 
restriction to $\,\Sigma\,$ of the homo\-mor\-phism (c) in 
Section~\ref{ro} is injective, making $\,\Sigma\,$ Abel\-i\-an. Now (a) in
Section~\ref{ro} implies that the image of $\,\Sigma'$ under the projection
$\,(r,u)\mapsto u\,$ spans a vector sub\-space $\,\lz\subseteq\xe$
satisfying condition (\ref{ace}-D), and so Remark~\ref{lagrg} gives 
$\,\dim\lz\le n-2$. Due to \hbox{(\ref{ace}-D)} and (a) in Section~\ref{ro}, 
$\,\hp'\nh=\{(1,0,\mathrm{Id})\}\times\bbR\times\lz\,$ {\it is an Abel\-i\-an 
subgroup of\/ $\,\hp$, containing\/ $\,\Sigma$, and the group operation in\/ 
$\,\hp'$ identified with\/ $\,\bbR\times\lz\,$ coincides with the addition in
the vector space\/} $\,\bbR\times\lz$. 

At the same time, the (necessarily compact) fibre of the bundle (\ref{dsc})
over the $\,\bbZ$-co\-set of $\,(\log t)/(\log q)\,$ is obviously the quotient 
$\,M\nnh_t\w=[\{t\}\times\bbR\times\mv]/\Sigma$. Compactness of
$\,M\nnh_t\w$ implies surjectivity of the
linear operator $\,\lz\ni u\mapsto u(t)\in\mv\hs$ for every
$\,t\in(0,\infty)$, since otherwise a nonzero linear functional vanishing on
its image, composed with the projection
$\,\{t\}\times\bbR\times\mv\nh\to\mv\nh$, would descend -- according to
(b) in Section~\ref{ro} -- to an unbounded function $\,M\nnh_t\w\to\bbR$.
Thus, $\,\dim\lz\ge n-2=\dim\mv$ which, due to the opposite
inequality in the last paragraph, gives both (\ref{ace}-A) and
(\ref{ace}-E). Remark~\ref{eqdif} with $\,I\nh=(0,\infty)\,$ and the
italicized conclusion of the
preceding paragraph, combined with compactness of each of the quotients  
$\,M\nnh_t\w$ (and the obvious proper
discontinuity of the action of $\,\Sigma\,$ on $\,\{t\}\times\bbR\times\mv$)
show that $\,\Sigma'$ is a lattice in $\,\bbR\times\lz$.

Finally, according to Remark~\ref{cnjug}, the right-hand side of (\ref{pru})
describes the conjugation by our $\,\hga\,$ in (\ref{gen}) applied to
$\,(1,0,\mathrm{Id},r,u)\in\Sigma$, which we identify here with 
$\,(r,u)$. As this conjugation obviously sends the kernel $\,\Sigma\,$ onto 
itself, we get (\ref{ace}-C), and so 
$\,\varPi\nh(\bbR\times\lz)=\bbR\times\lz\,$ (since $\,\Sigma'$ is a 
lattice in $\,\bbR\times\lz$). Now (\ref{pru}) yields (\ref{ace}-B), 
which completes the proof.
\end{proof}
\begin{lemma}\label{cnseq}Under the hypotheses preceding\/ {\rm(\ref{pru})},
let a vector sub\-space\/ $\,\lz\subseteq\xe$ satisfy\/ {\rm(\ref{ace}-A)}
-- {\rm(\ref{ace}-C)}, a basis\/ $\,u\nh^+_1\nnh,u^-_1,\dots,u\nh^+_m,u^-_m$ of\/
$\,\xe^\bbC$ containing a basis\/ $\,u_1\w,\dots,u_m\w$ of\/ $\,\lz^\bbC$ be
chosen as in Theorem\/~{\rm\ref{spctr}}, and\/ $\,\gy_1\w,\dots,\gy_m\w$ be
the corresponding complex characteristic roots of\/
$\,\bc\hh T:\xe\nh\to\xe\hs$ selected from\/
$\,\gy\nh^+_1,\gy^-_1,\dots,\gy\nh^+_m,\gy^-_m$ given by\/
{\rm(\ref{upm})}. Then
\begin{enumerate}
\item[(i)] $\gy_0\w=q\nh^{-\nnh1}$ and\/ $\,\gy_1\w,\dots,\gy_m\w$ form a\/ 
$\,\mathrm{GL}\hh(\bbZ)$-spec\-trum,
\end{enumerate}
in the sense that they are the complex roots of some\/
$\,\mathrm{GL}\hh(\bbZ)$-pol\-y\-no\-mi\-al of degree\/ $\,m+1$, defined as
in Section\/~{\rm\ref{gp}}, and
\begin{enumerate}
\item[(ii)] the product\/ $\,\gy_1\w\nh\ldots\gy_m\w$ equals\/ $\,q\, $ or\/
$\,-q$.
\end{enumerate}
Furthermore, assuming in addition that
\begin{enumerate}
\item[(iii)] one of\/ $\,\mu^\pm$ is a power of\/ $\,q\,$ with a rational 
exponent,
\end{enumerate}
we have the following conclusions.
\begin{enumerate}
\item[(iv)] Both\/ $\,\mu^\pm$ are powers of\/ $\,q\,$ with integer exponents.
\item[(v)] $\gy\nh^+_1,\gy^-_1,\dots,\gy\nh^+_m,\gy^-_m$ are all distinct,
real and positive.
\item[(vi)] Exactly one of\/ $\,\gy_1\w,\dots,\gy_m\w$ equals\/ $\,q$.
\item[(vii)] Just one, or none of\/ $\,\gy_1\w,\dots,\gy_m\w$ equals\/ $\,1\,$
if\/ $\,n\,$ is even, or odd.
\item[(viii)] Those\/ $\,\gy_1\w,\dots,\gy_m\w$ not equal to\/ $\,q\,$ or\/
$\,1\,$ form pairs of mutual inverses.
\item[(ix)] $\varOmega(u\nh^\pm_i,u\nh^\pm_j)=0\,$ for all\/
$\,i,j\in\{1,\dots,m\}\,$ and both signs\/ $\,\pm\hs$.
\item[(x)] $\varOmega(u\nh^\pm_i,u\nh^\mp_j)\ne0\,$ if and only if\/ 
$\,i+j=m+1$,
\end{enumerate}
\end{lemma}
\begin{proof}Assertion (i) is immediate from (\ref{pru}) and (\ref{ace}-C)
along with (\ref{glz}), and (ii) from (i). Assuming (iii), we see -- using 
(\ref{upm}), (\ref{mpq}) and (\ref{drg}) -- that, for the
$\,\mathrm{GL}\hh(\bbZ)$-pol\-y\-no\-mi\-al 
$\,P\hs$ with the roots $\,\gy_0\w,\dots,\gy_m\w$,
\begin{enumerate}
\item[(xi)] the irreducible factors of $\,P\hs$ must all be linear or
quadratic,
\end{enumerate}
higher degree cyclo\-tom\-ic polynomials being excluded since the roots are
all real. Thus, one of
$\,\gy_1\w,\dots,\gy_m\w$ equals $\,q$, to match
$\,\gy_0\w=q\nh^{-\nnh1}\nnh$, and (\ref{upm}) combined with (\ref{mpq})
yields (iv). Since $\,|\hh\gy^\pm_j|$ is, for either sign $\,\pm\hs$, a
strictly monotone function of $\,j$, to prove (v) it suffices to consider
the case $\,q^{m+1-2\hn j}\mu^\pm\nh=q^{m+1-2i}\mu^\mp\nh$, that is,
$\,\mu^\pm\nnh/\mu^\mp\nh=q^{2(j-i)}\nh$. Multiplied by
$\,\mu\nh^\pm\nnh\mu^\mp\nh=q\nh^{-\nnh1}\nh$, cf.\ (\ref{mpq}), this makes 
$\,(\mu\nh^\pm)^2$ a power of $\,q\,$ with an {\it odd\/} integer exponent,
contrary to (iv), so that (v) follows. From (iii) and (xi) we now get 
(viii).

For our basis $\,u\nh^\pm_j$ of $\,\xe\nh$,
di\-ag\-o\-nal\-iz\-ing $\,\bc\hh T\,$ with the eigen\-values
$\,\gy^\pm_j\nnh=\nh q^{m+1-2\hn j}\mu^\pm\nh$,
(g) in Section~\ref{ro} gives
\[
\begin{array}{l}
q\nh^{-\nnh1}\nh\varOmega(u\nh^\pm_i,u\nh^\pm_j)
=q^{2m+2-2i-2\hn j}(\mu^\pm)^2\varOmega(u\nh^\pm_i,u\nh^\pm_j)\hh,\\
q\nh^{-\nnh1}\nh\varOmega(u\nh^\pm_i,u\nh^\mp_j)
=q^{2m+2-2i-2\hn j}\mu\nh^+\nh\mu\hn^-\varOmega(u\nh^\pm_i,u\nh^\mp_j)\hh.
\end{array}
\]
Thus, the inequality $\,\varOmega(u\nh^\pm_i,u\nh^\pm_j)\ne0\,$ would, again,
make $\,(\mu\nh^\pm)^2$ a power of $\,q\,$ with an odd integer exponent, 
contradicting (iv), which yields (ix). Similarly, assuming that
$\,\varOmega(u\nh^\pm_i,u\nh^\mp_j)\ne0$, we now get, from (\ref{mpq}),
$\,i+j=m+1$. The converse 
implication needed in (x) follows, via (ix), from nondegeneracy of
$\,\varOmega$.
\end{proof}
\begin{lemma}\label{abcde}With the assumptions and notations of
Lemma\/~{\rm\ref{cnseq}}, let $\,\lz\,$ this time satisfy all of\/
{\rm(\ref{ace})}. Then conditions\/ {\rm(i)} -- {\rm(x)} in
Lemma\/~{\rm\ref{cnseq}} all hold, so that\/ $\,\mu^\pm$ and\/ 
$\,\gy^\pm_j$ are all real, while
\begin{enumerate}
\item[(i)] the number of pluses is different from that of minuses
\end{enumerate}
among the\/ $\,\pm\,$ superscripts of those
$\,\gy\nh^+_1,\gy^-_1,\dots,\gy\nh^+_m,\gy^-_m$ which form the characteristic
roots\/ $\,\gy_1\w,\dots,\gy_m\w$ of\/ $\,\bc\hh T:\lz\to\lz$. Finally,
for the basis\/ $\,\cs=\{u_1\w,\dots,u_m\w\}$ of\/ $\,\lz$ contained in\/ 
the basis\/ $\,\{u\nh^+_1\nnh,u^-_1,\dots,u\nh^+_m,u^-_m\}\,$ of\/ 
$\,\xe\nh$, with\/ $\,|\hskip2.3pt|\,$ denoting cardinality,
\begin{enumerate}
\item[(ii)] $|\cs\cap\{u\nh^+_1\nnh,u^-_1,\dots,u\nh^+_j,u^-_j\}|\le j\,$
whenever\/ $\,j=1,\dots,m$,
\item[(iii)] $|\cs\cap\{u\nh^+_i,u^-_j\}|=1\,$ if\/
$\,i,j\in\{1,\dots,m\}\,$ and\/ $\,i+j=m+1$.
\end{enumerate}
\end{lemma}
\begin{proof}If (ii) failed to hold, the evaluation operator in (\ref{ace}-E),
complexified if necessary, would send $\,\{u\hn_1\w,\dots,u\nh_{j+1}\w\}\,$
into the span of the vectors $\,e\nh_1\w,\dots,e\nnh_j\w$ appearing in
(\ref{upm}), contrary to its injectivity. From (ii) we obtain
\begin{enumerate}
\item[(iv)] $k(j)\ge j\,$ for all $\,j=1,\dots,m$,
\end{enumerate}
$k(j)\in\{1,\dots,m\}\,$ being such that $\,u_j\w=u\nh^\pm_{k(j)}$ with some
sign $\,\pm\hs$, since, otherwise, 
$\,\cs\cap\{u\nh^+_1\nnh,u^-_1,\dots,u\nh^+_{k(j)},u^-_{k(j)}\}\,$ would
have at least $\,j>k(j)\,$ elements.

To prove (i), we now assume its negation, and evaluate the product of those 
$\,\gy^\pm_j\nnh=\nh q^{m+1-2\hn j}\mu^\pm$ 
in (\ref{upm}) which constitute $\,\gy_1\w,\dots,\gy_m\w$. Both factors 
$\,\mu\nh^+\nh,\mu\hn^-$ appear in this product the same number of times,
$\,m/2$, which makes $\,m\,$ even, and by (\ref{mpq}) their occurrences
contribute to our product $\,\gy_1\w\nh\ldots\gy_m\w$ a total factor of
$\,q^{-m/2}\nh$. On the other hand, the set 
$\,\{q^{m+1-2\hn j}:1\le j\le m\}=\{q^{m-1}\nh,q^{m-3}\nh,\dots,q^{1-m}\}$ is 
closed under taking inverses, so that $\,\prod_{j=1}^mq^{m+1-2\hn j}=1$. 
Writing $\,k(j)=j+\ell(j)$, with $\,\ell(j)\ge0\,$ due to (iv), we now have
\begin{equation}\label{lje}
\gy_j\w=\gy^\pm_{k(j)}\nh=\,q^{m+1-2k(j)}\mu^\pm\nh
=\,q^{m+1-2\hn j}\mu^\pm q^{-\nh2\ell(j)}\nh,
\end{equation}
making $\,\gy_1\w\nh\ldots\gy_m\w$ equal to $\,1\,$ times $\,q^{-m/2}$
times $\,\prod_{j=1}^mq^{-\nh2\ell(j)}\nh$, that is, a power of $\,q\,$ with a 
negative exponent, contrary to Lemma~\ref{cnseq}(ii).

Next, (i) implies that $\,\mu^\pm$ and $\,\gy^\pm_j$ are all real, for
otherwise $\,\gy_j\w$ in (\ref{lje}), forming along with
$\,\gy_0\w=q\nh^{-\nnh1}$ the spectrum of a real matrix, would come in
non\-real conjugate pairs, with the same number of positive real parts as
negative ones. Thus, by (\ref{wlo}), $\,\mu^\pm\nh>0$. 
Using (i) and reality of $\,\mu^\pm$ we now evaluate the
product $\,\gy_1\w\nh\ldots\gy_m\w=\pm q$ in Lemma~\ref{cnseq}(ii),
observing that not all $\,\mu\nh^+\nnh,\mu^-$ undergo pairwise
``cancellations'' (forming the product $\,q\nh^{-\nnh1}$), but instead
Lemma~\ref{cnseq}(ii) equates some power of $\,\mu\nh^+$ or $\,\mu^-\nh$, with
a positive integer exponent, to a power of $\,q$, and so positivity of
$\,\mu^\pm$ yields condition (iii) in Lemma~\ref{cnseq}, which in turn implies
(iv) -- (x).

Finally, the $\,m$-el\-e\-ment family 
$\,\mathcal{P}\nh=\{\{u\nh^+_i\nnh,u^-_j\}:i+j=m+1\}$ forms a partition
of $\,\{u\nh^+_1\nnh,u^-_1,\dots,u\nh^+_m,u^-_m\}\,$ into disjoint
two-el\-e\-ment sub\-sets, while the mapping \hbox{$F:\cs\to\mathcal{P}\hs$} 
given by $\,u\in F(u)\,$ is injective: 
$\,|\cs\cap\{u\nh^+_i,u^-_j\}|\le1\,$ if $\,i+j=m+1$, or else 
Lemma~\ref{cnseq}(x) would contradict (\ref{ace}-D). As $\,|\cs|=m$,
surjectivity of $\,F$ thus follows, proving (iii).
\end{proof}
We now complete the proof of Theorem~\ref{modif} by observing that a vector
sub\-space $\,\lz\subseteq\xe\hs$ with (\ref{ace}) gives rise to a sub\-set 
$\,\zs\,$ of $\,\vr\nh=\{1,\dots,2m\}$, for $\,m=n-2$, satisfying conditions
(a) -- (e) in Theorem~\ref{noset}, which -- according to Theorem~\ref{noset}
-- cannot exist. Namely, using Lemma~\ref{cnseq}(iv) we define
$\,\r\in\bbZ\,$ by $\,\mu\nh^+\nh=q^\r\nh$, so that, by (\ref{mpq}),
$\,\mu^-\nh=q\nh^{-\hn\r-1}\nh$.
Next, the obvious or\-der-pre\-serv\-ing bijection
\begin{equation}\label{bij}
\vr\nh=\{1,\dots,2m\}\,\to\,\{u\nh^+_1\nnh,u^-_1,\dots,u\nh^+_m,u^-_m\}
\end{equation}
(notation of Lemma~\ref{cnseq}) which, explicitly, sends $\,a\in\vr\hs$ to
$\,u^-_j$ when $\,a=2\hn j$ is even, or to $\,u\nh^+_j$ for odd
$\,a=2\hn j-1$, is used from now on to identify the two sets, and we declare
$\,\zs\,$ to be the sub\-set of $\,\vr\hs$ corresponding under (\ref{bij}) 
to the basis $\,\cs=\{u_1\w,\dots,u_m\w\}$ of $\,\lz$. The function assigning
to each $\,u\nh^\pm_j$ the corresponding eigen\-value 
$\,\gy^\pm_j=q^{m+1-2\hn j}\mu^\pm$ treated, via (\ref{bij}), as defined on
$\,\vr\nh$, is now easily seen to be given by
$\,\vr\ni a\mapsto q^{\ep(a)}\nh$, with (\ref{aph}-i). Referring to (a) -- (e)
in Theorem~\ref{noset} simply as (a) -- (e), we observe that assertions (ii)
and (iii) of Lemma~\ref{abcde} yield (e) and (c), while (b), the first 
claim in (a) and (d) trivially follow from Lemma~\ref{cnseq}(vi)\hh-\hh(viii)
(the latter guaranteed to hold by Lemma~\ref{abcde}). Finally, the relation
$\,\vf(a_1\w)\notin\zs$ in (a) which, in view of (\ref{aph}-iii) and
(\ref{aph}-v), amounts to 
$\,q\nh^{-\nnh1}\nnh\notin\{\gy_1\w,\dots,\gy_m\w\}$, is thus immediate
since otherwise, due to Lemma~\ref{cnseq}(viii), the inverse $\,q\,$ of
$\,q\nh^{-\nnh1}$ would occur on the list $\,\gy_1\w,\dots,\gy_m\w$ {\it
twice}, contradicting Lemma~\ref{cnseq}(v).


\begin{thebibliography}{99}

\bibitem{derdzinski-80}A.\hskip2.3ptDerdzi\'nski, {\em On con\-for\-mal\-ly
symmetric Ric\-ci-re\-cur\-rent manifolds with Abel\-i\-an fundamental
groups}, Tensor (N.\hskip1.6ptS.) \textbf{34} (1980), 21--29.

\bibitem{derdzinski-roter-77}A.\hskip2.3ptDerdzi\'nski and W.\hskip2.3ptRoter,
{\em On con\-for\-mal\-ly symmetric manifolds with metrics of indices $\,0\,$
and $\,1\hh$}, Tensor (N.\hskip1.6ptS.) \textbf{31} (1977), 255--259.

\bibitem{derdzinski-roter-07}A.\hskip2.3ptDerdzinski and W.\hskip2.3ptRoter, 
{\em Global properties of indeﬁnite metrics with parallel Weyl tensor},
in: Pure and Applied Differential Geometry - PADGE 2007, eds.\ F.\ Dillen and
I.\ Van de Woestyne, Berichte aus der Mathematik, Shaker Verlag, Aachen, 2007,
63--72.

\bibitem{derdzinski-roter-09}A.\hskip2.3ptDerdzinski and W.\hskip2.3ptRoter, 
{\em The local structure of con\-for\-mal\-ly symmetric manifolds}, Bull.\
Belgian Math.\ Soc. \textbf{16} (2009), 117--128.

\bibitem{derdzinski-roter-10}A.\hskip2.3ptDerdzinski and
W\nnh.\hskip2.3ptRoter, 
{\em Compact pseu\-\hbox{do\hskip.7pt-}Riem\-ann\-i\-an manifolds with
parallel Weyl tensor}, Ann.\ Glob.\ Anal.\ Geom. \textbf{37} (2010), 73--90.

\bibitem{derdzinski-terek-ne}A.\hskip2.3ptDerdzinski and I.\hskip2.3ptTerek,
{\em New examples of compact Weyl-par\-al\-lel manifolds}, preprint (available 
from https:/\hskip-1pt/arxiv.org/pdf/2210.03660.pdf).

\bibitem{derdzinski-terek-tc}A.\hskip2.3ptDerdzinski and I.\hskip2.3ptTerek,
{\em The topology of compact rank-one ECS manifolds}, preprint (available 
from https:/\hskip-1pt/arxiv.org/pdf/2210.09195.pdf).

\bibitem{derdzinski-terek-ms}A.\hskip2.3ptDerdzinski and I.\hskip2.3ptTerek,
{\em The metric structure of compact rank-one ECS manifolds}, pre\-print
(available from https:/\hskip-1pt/arxiv.org/pdf/2304.10388.pdf).

\bibitem{derdzinski-terek-cl}A.\hskip2.3ptDerdzinski and I.\hskip2.3ptTerek,
{\em Compact locally homogeneous manifolds with parallel Weyl tensor},
preprint (available from
https:/\hskip-1pt/arxiv.org/pdf/2306.01600.pdf).

\bibitem{dundas}B.\hskip2.3ptI.\hskip2.3ptDundas, {\em A Short Course in 
Differential Topology}, Cambridge Mathematical Textbooks. Cambridge University 
Press, Cambridge, 2018.

\bibitem{maier}H.\hskip2.3ptMaier, {\em Anatomy of integers and
cyclo\-tom\-ic polynomials}, in: Anatomy of Integers (J.-M.\hskip2.3ptDe
Koninck, A.\hskip2.3ptGranville, and F.\hskip2.3ptLuca, eds.), CRM Proceedings
\&\ Lecture Notes, vol. \textbf{46}, American Mathematical Society,
Providence, RI, 2008, pp. 89--95.

\bibitem{olszak}Z.\hskip2.3ptOlszak, {\em On con\-for\-mal\-ly recurrent
manifolds, I\hs{\rm:} Special distributions}, Zesz.\ Nauk.\ Po\-li\-tech.\ 
\'Sl., Mat.-Fiz. \textbf{68} (1993), 213--225.

\bibitem{roter}W\nnh.\hskip2.3ptRoter, {\em On con\-for\-mal\-ly symmetric 
Ric\-ci-re\-cur\-rent spaces}, Colloq.\ Math. \textbf{31} (1974), 87--96.

\end{thebibliography}
\end{document}